\documentclass[11pt]{article}
\usepackage{amsthm,amsmath,amssymb,amsfonts,bbm,graphics,a4wide,rotating,graphicx}
\usepackage[caption=false,font=footnotesize]{subfig} 
\usepackage[usenames]{color}
\definecolor{plum}  {rgb}{.6,0,.6}
\definecolor{forest}  {rgb}{0,.7,0}
\definecolor{midnight}  {rgb}{0,0,.7}
\usepackage[colorlinks=true,linkcolor=midnight,pdftex,citecolor=forest]{hyperref}
\def\eps{\varepsilon}
\def\ka{\kappa_2}
\def\kb{\kappa_1}
\def\ybar{\overline{Y}}
\def\ones{\mathbbm{1}}

\def\reals{{\mathbb R}}

\def\expect{{\mathbb E}} 
\def\eg{{\em e.g.,~}}
\def\cf{{\em cf.~}}

\DeclareMathOperator{\supp}{supp}

\def\Var{\mathrm{Var}}

\DeclareMathOperator*{\argmin}{argmin}

\newtheorem{Def}{Definition}
\newtheorem{proposition}{Proposition}
\newtheorem{remark}{Remark}
\newtheorem{lemma}{Lemma}

\renewenvironment{proof}[1]{\noindent{\bf Proof of #1.}}{
 \hfill $\square$
  }

\newtheoremstyle{assumption}{6pt}{6pt}{\rm}{}{\sffamily}{ }{ }{}
\theoremstyle{assumption}
\newtheorem*{assumption*}{\assumptionnumber}
\providecommand{\assumptionnumber}{}
\makeatletter

\newenvironment{pHyp}[2]
 {%
  \renewcommand{\assumptionnumber}{\sc Assumption #1($#2$)}%
  \begin{assumption*}%
  \protected@edef\@currentlabel{#1}%
 }
 {%
  \end{assumption*}
 }
\makeatother
\newcommand{\asref}[2]{\ref{#1}($#2$)}

\def\tA{\widetilde{A}}

\def\hd{d}
\def\tG{\widetilde{G}}

\def\xls{\wh{x}^{\rm OLS}}
\def\xmle{\wh{x}^{\rm MLE}}
\def\xL{\wh{x}^{\rm LAS
SO}}

\def\xW{\wh{x}^{\rm WL}}

\def\dmin{d_{\min}}
\def\dmax{d_{\max}}
\def\nObs{n}
\def\nU{m}

\newcommand{\norm}[1]{\ensuremath{\vert\!\vert #1 \vert\!\vert}}
\newcommand{\E}{\ensuremath{\mathbb{E}}}

\renewcommand{\P}{\ensuremath{\mathbb{P}}}

\newcommand{\R}{\ensuremath{\mathbb{R}}}

\newcommand{\wh}[1]{\widehat{#1}}
\newcommand{\wt}[1]{\widetilde{#1}}

\newcommand{\I}{{I}}

\newcommand{\1}{{\bf 1}}

\definecolor{darkred}{rgb}{.8,.0,.035}

\newcommand{\revision}[1]{\textcolor{black}{#1}}

\newcommand{\deq}{:=}
\numberwithin{equation}{section}

\newcommand{\T}{\top}

\def\tA{\wt{A}}

\def\tY{\wt{Y}}

\def\Nb{\mathbb{N}}
\def\Nu{\mathbb{U}}

\newcommand{\ave}[1]{\langle #1 \rangle}

\begin{document}
\title{A data-dependent weighted LASSO under Poisson noise}
\date{} 

\author{{
\sc Xin J. Hunt},\\[2pt]
SAS Institute Inc., Cary, NC USA\\[6pt]
{\sc Patricia Reynaud-Bouret}\\[2pt]
University of C\^{o}te d'Azur, CNRS, LJAD, Nice, France\\[6pt]
{\sc Vincent Rivoirard} \\[2pt]
University of Paris-Dauphine, Paris, France\\[6pt]
{\sc Laure Sansonnet} \\[2pt]
INRA - AgroParisTech, Paris, France\\[6pt]
{\sc Rebecca Willett}$^*$ \\[2pt]
University of Wisconsin-Madison, Madison, WI, USA}

\maketitle

\begin{abstract} {Sparse linear inverse problems appear in a variety
    of settings, but often the noise contaminating observations cannot
    accurately be described as bounded by or arising from a Gaussian
    distribution. Poisson observations in particular are a
    characteristic feature of several real-world
    applications. Previous work on sparse Poisson inverse problems
    encountered several limiting technical hurdles. This paper
    describes a novel alternative analysis approach for sparse Poisson
    inverse problems that (a) sidesteps the technical challenges
    present in previous work, (b) admits estimators that can readily
    be computed using off-the-shelf LASSO algorithms, and (c) hints at
    a general framework for broad classes of
    noise in sparse linear inverse problems. At the heart of this new approach lies a weighted LASSO
    estimator for which {\em data-dependent} weights are based on Poisson
    concentration inequalities. Unlike previous analyses of the
    weighted LASSO, the proposed analysis depends on conditions which
    can be checked or shown to hold in general settings with high
    probability.}
  2000 Math Subject Classification: 60E15, 62G05, 62G08, 94A12
\end{abstract}

\section{Introduction}
Poisson noise arises in a wide variety of applications and settings,
including PET, SPECT, and pediatric or spectral CT
\cite{willett:tmi03,fesslerPCS,schmidt2009optimal} in medical imaging,
x-ray astronomy \cite{Kazik14,Kazik13,starckPCS}, genomics
\cite{laure}, network packet analysis
\cite{ElephantsMice,CounterBraids}, crime rate analysis
\cite{PoissonCrime}, and social media analysis \cite{socioscope}. In
these and other settings, observations are characterized by discrete
counts of events (\eg photons hitting a detector or packets arriving
at a network router), and our task is to infer the underlying signal
or system even when the number of observed events is very small.
Methods for solving Poisson inverse problems have been studied using a
variety of mathematical tools, with recent efforts focused on
leveraging signal sparsity 
\cite{PCS_garvesh,pcs,expFamCS,expander_PCS,fesslerPCS,starckPCS,rohban2015minimax,IPR}.

Unfortunately, the development of risk bounds for sparse Poisson
inverse problems presents some significant technical
challenges. Methods that rely on the negative Poisson log-likelihood to
measure how well an estimate fits observed data perform well in
practice but are challenging to analyze. For example, the analysis
framework considered in \cite{PCS_garvesh,pcs,expFamCS}
builds upon a coding-theoretic bound which is difficult to adapt to
many of the computationally tractable sparsity regularizers used in
the Least Absolute Shrinkage and Selection Operator (LASSO)
\cite{LASSO} or Compressed Sensing (CS) \cite{CS:candes2,CS:donoho};
those analyses have been based on impractical $\ell_0$ sparsity
regularizers.  In contrast, the standard LASSO analysis framework
easily handles a variety of regularization methods and has been
generalized in several directions \cite{BRT, bunea, Lounici, LASSO,
  van2008high, Buhlmann:2011}. However, it does not account for
Poisson noise, which is heterogeneous and dependent on the unknown
signal to be estimated.

This paper presents an alternative approach that sidesteps these
challenges. We describe a novel weighted LASSO estimator, where the
data-dependent weights used in the regularizer are based on Poisson
concentration inequalities and control for the ill-posedness of the
inverse problem and heteroscedastic noise simultaneously.  We
establish oracle inequalities and recovery error bounds for general
settings, and then explore the nuances of our approach within two
specific sparse Poisson inverse problems arising in genomics and
imaging.

\subsection{Problem formulation \label{sec:pbform}}
We observe a potentially random matrix
$A=(a_{k,l})_{k,l} \in \reals_+^{\nObs \times p}$ and conditionally on $A$, we
observe 
\begin{equation}
Y \sim \mathcal{P}(A x^*)
\label{eq:obs}
\end{equation}
 where $Y \in \reals_+^\nObs$,
$x^* \in \reals_+^p$, and where $x^*$ is sparse or compressible. The
notation $\cal{P}$ denotes the Poisson distribution, so that,
conditioned on $A$ and $x^*$, we have the likelihood
$$p(Y_k | Ax^*) = e^{-(Ax^*)_k} [(Ax^*)_k]^{Y_k} / Y_k!, \qquad k = 1,\ldots,\nObs.$$
Conditioned on $Ax^*$, the elements of $Y$ are independent.  The aim
is to recover $x^*$, the true signal of interest. The matrix $A$ corresponds to a
sensing matrix or operator which linearly projects $x^*$ into another
space before we collect Poisson observations. Often we will have $\nObs<
p$, but this inverse problem can still be challenging if $\nObs \ge p$
depending on the signal-to-noise ratio or the condition of the
operator $A$.

Because elements of $A$ are nonnegative, we cannot rely on the
standard assumption that $A^\T A$ is ``close to'' an identity
matrix. However, in many settings there is a proxy operator, denoted
$\tA$, which is amenable to sparse inverse problems and is a simple
linear transformation of the original operator $A$. A complementary
linear transformation may then be applied to $Y$ to generate proxy
observations $\tY$, and we use $\tA$ and $\tY$ in the estimators
defined below. In general, the linear transformations are
problem-dependent and should be chosen to ensure
our main assumptions (presented
in Section~\ref{sec:bounds}) are satisfied.  We provide explicit examples in
Sections~\ref{sec:bernoulli} and~\ref{sec:conv}.  Note that other
preconditioning transformations have been proposed in the
literature. See for instance \cite{PBHT, Wauthier_acomparative, HJ}
where various procedures are suggested but very different in spirit
from ours.

\subsection{Weighted LASSO estimator for Poisson inverse problems}
\revision{The basic idea of our approach is the
following. We consider two main estimation methods in this paper:
\begin{description}
\item[{\bf (Classical) LASSO estimator:}]
\begin{equation}
\xL := \argmin_{x \in \reals^p} \left\{\|\tY-\tA x\|_{2}^2 +
  \gamma d \|x\|_1 \right\},
\label{eq:unweighted}
\end{equation}
where $\gamma > 2$ is a constant and $d > 0$ is a data-dependent
scalar to be defined later.
\item[{\bf Weighted LASSO estimator:}]
\begin{equation}
\xW := \argmin_{x \in \reals^p} \left\{\|\tY-\tA x\|_{2}^2 +
  \gamma \sum_{k=1}^{p} d_k|x_k|\right\}
\label{eq:lasso}
\end{equation}
where $\gamma > 2$ is a constant and the $d_k$'s are positive and
data-dependent; they will be defined later.
Note that the estimator in \eqref{eq:lasso} can equivalently
be written as
\begin{subequations}
\begin{align} 
  \wh{z} =& \argmin_{z \in \reals^p} \left\{\|\tY-\tA D^{-1}
    z\|_{2}^2 +
    \gamma \|z\|_1\right\} \label{eq:wlopt} \\
  \xW =& D^{-1} \wh{z}
\end{align}
\end{subequations}
where $D$ is a diagonal matrix with the $k^{\rm th}$ diagonal element
equal to $d_k$. Note that the optimization problem in \eqref{eq:wlopt} can be
solved efficiently using off-the-shelf LASSO solvers. Since $z$ and $D^{-1}z$ will always have the same
support, this formulation suggests that the weighted LASSO estimator
in \eqref{eq:lasso} is essentially a data-dependent reweighing of the
columns of $\tA $.
\end{description}}

\revision{A weighted LASSO estimator similar to~\eqref{eq:lasso} has been
proposed and analyzed in past work, notably \cite{vGBZ11,BRT,Juditsky}, where
the weights are considered fixed and arbitrary. The analysis in
\cite{vGBZ11}, however, does not extend to signal-dependent noise
(as we have in Poisson noise settings). In addition, risk bounds in
that work hinge on a certain ``{\em weighted} irrepresentable
condition'' on the sensing or design matrix $\tA$ which cannot be
verified or guaranteed for the data-dependent weights we consider,
even when $\tA$ is known to satisfy criteria such as the Restricted
Eigenvalue condition \cite{BRT} or Restricted Isometry Property
\cite{RIPCandes}. 
An  analysis of the weighted LASSO estimator using standard
  LASSO bounding techniques (\cf
  \cite{Juditsky,negahban2009unified}) yields looser
  bounds than those presented below.}  

\revision{If $x^*$ has support $S^*$ of size $s:=|S^*|$ and if we choose
weights $d_1,\ldots,d_p$ satisfying
\begin{equation}
  |(\tA^\T (\tY-\tA x^*))_k| \le d_k \qquad \mbox{for } k = 1,\ldots,
  p,
\label{eq:condition}
\end{equation}
then, over an appropriate range of
values of $s$, the following risk bounds hold for the LASSO and weighted LASSO estimates
under conditions of Proposition~\ref{th:L2}:
\begin{subequations}
\begin{align}
\|\xW-x^*\|_2^2
\le&\; \frac{\rho_\gamma^2}{\eta}\sum_{k \in S^*} d_k^2
\label{eq:boundWL}\\
\|\xL-x^*\|_2^2
\le&\; \frac{\rho_\gamma^2}{\eta}sd^2,
\label{eq:boundL}
\end{align}
\label{eq:twobounds}
\end{subequations}
\noindent where $\rho_\gamma$ only depends on $\gamma$ and $\eta$ is a parameter
associated with the restricted eigenvalue condition of the sensing
matrix $\tilde{A}$ (see Proposition \ref{th:L2}). Note that the
condition in \eqref{eq:condition} is similar to the constraint in the
Dantzig selector \cite{candes2007dantzig}.  The bounds in
\eqref{eq:twobounds} highlight that if we did not have
practical constraints such as the fact that the $d_k$'s should only
depend on the data, one could take $d_k= |(\tA^\T (\tY-\tA x^*))_k|$
for the weighted LASSO, whereas for the LASSO, one could only take
$d=\max_k{|(\tA^\T (\tY-\tA x^*))_k|}$, which only leads to worse
bounds.} 

Furthermore, we consider an oracle estimator that consists of least
squares estimation on the true support $S^*$,
$$\xls \deq  I_{S^*} (\tA _{S^*})^\# \tY,$$
and show that 
$$  \sum_{k \in S^*} (\tA^\T (\tY-\tA x^*))_k^2 \lesssim \| \xls-x^*\|_2^2\lesssim \sum_{k \in S^*} (\tA^\T (\tY-\tA x^*))_k^2.$$
If each $d_k$ in the weighted LASSO estimator is close to
$|(\tA^\T (\tY-\tA x^*))_k|$, then the bounds on the weighted LASSO
are close to those of the oracle least squares estimator.

In practice the weights can only depend on the observed data. We show
two examples, a Bernoulli sensing matrix and random convolution (see Sections \ref{sec:bernoulli} and \ref{sec:conv}), in which we can
compute weights from the data (by using Poisson concentration
inequalities) such that \eqref{eq:condition} holds with high
probability {\em and} those weights are small enough to ensure risk
bounds that have a better convergence rate than LASSO estimates.

\subsection{The role of the weights}
Our approach, where the weights in our regularizer are random
variables, is similar to\,\cite{BLPR11,HRBR12,HMZ08,Zou06}.  In some
sense, the weights play the same role as the thresholds in the
estimation procedure proposed
in\,\cite{DJ94,JLL04,RBR10,RBRTM11,laure}. The role of the weights are
twofold:
\begin{itemize}
\item control of the random fluctuations of $\tA^\T\tY$
  around its mean, and 
\item compensate for the ill-posedness due to $\tA$. Note that
  ill-posedness is strengthened by the heteroscedasticity of the
  Poisson noise.
\end{itemize}

To better understand the role of the weights, let us look at a
toy example where $A$ is a diagonal matrix with decreasing
eigenvalues $\lambda_1 > \cdots > \lambda_p>0$ to which we add heteroscedastic noise. 
Here one could rephrase the direct problem as
\begin{equation}
y_k= \lambda_k x^*_k + \epsilon_k,
\label{eq:diag}
\end{equation}
where $\epsilon_k$ has zero mean and standard deviation $\sigma_k$.
This toy example is often derived via a diagonalization of some
inverse problem via the singular value decomposition of the matrix
$A$\footnote{Specifically, if the SVD of $A$ is $U\Lambda V^\top$,
  then observations of the form $y = Ax^* + \epsilon$ can be
  equivalently expressed as
  $U^\top y = \Lambda (V^\top x^*) + U^\top \epsilon$, yielding
  measurements of the form \eqref{eq:diag}.}.  The assumptions of
Poisson noise and that $x^*$ is sparse do not generally hold if we
diagonalize the problem in \eqref{eq:obs}, so the model in
\eqref{eq:diag}
should not be seen as a sketch of what can be done in general but more
as an illustration to understand the main tools that are used in the
sequel.

\revision{First of all, note that if we know the support of $x^*$, the
  least-square estimator, $\hat{x}^{LS}$, is trivial here and amounts
  to $(\hat{x}^{LS})_k= y_k/\lambda_k$ if $k$ is in $S^*$ and $0$
  anywhere else. It is then easy to see that
$$\E(\|\hat{x}^{LS}-x^*\|^2) = \sum_{k\in S^*} \frac{\sigma^2_k}{\lambda_k^2}.$$
This is our benchmark\footnote{In addition, note that in a Gaussian
  framework, the least-square estimator is also the MLE and reaches
  Cramer-Rao bound. So in a certain sense, asymptotically it is the
  smallest risk we could hope for.}, our oracle in some sense, since
it cannot be computed without knowing the support of the true signal.}

The weighted LASSO method needs two main ingredients: the linear
transformation $(\tilde{A},\tilde{Y})$ and the weights satisfying
\eqref{eq:condition}.  Let us first consider the ramifications of
setting ${\tY}=Y$, ${\tA}=A$. The quantity $\eta$ appearing in
\eqref{eq:boundWL} is then connected to the smallest eigenvalue of
$A$, and it is easy to show\footnote{In this diagonal case, it is easy
  to see that Assumption~\ref{as:RE} holds with $\kappa_2=\lambda_p$
  and $\kappa_1=0$, which leads to $\epsilon=\kappa_2$ in Proposition
  \ref{th:L2}} that $\eta=\lambda_p^4$. On the other hand, by
\eqref{eq:condition}, $d_k$ should be an upper bound on
$\lambda_k\epsilon_k$. Heuristically, $d_k$ should therefore be of the
order of $\lambda_k\sigma_k$ and the bound \eqref{eq:boundWL} is then
on the order of
$\sum_{k\in S^*} \frac{\lambda_k^2\sigma_k^2}{\lambda_p^4}$.  In
particular, even if the true support $S^*$ coincides with indices
where the $\lambda_k$'s are large, we still see our rate controlled by
the potentially large factor of $\lambda_p^{-4}$.

On the other hand, the classical inverse problem choice
${\tY}=A^{-1}Y$, ${\tA}=A^{-1}A = I_p$ gives that $\eta=1$ and that
$d_k$ should be of the order of $\sigma_k/\lambda_k$ at
least\footnote{In practice, because there will be randomness to take
  into account, guaranteeing \eqref{eq:condition} for all $k$ with
  high probability, will lead to an extra $\log(p)$ factor here, that
  may be thought as the price to pay for adaptivity with respect to
  unknown support $S^*$ (see in particular the two main
  examples).}. Therefore the upper bound \eqref{eq:boundWL} is then of
the order of $\sum_{k\in S^*} \frac{\sigma_k^2}{\lambda_k^2}$, that
is, we reach the benchmark risk of the least-square estimator.  For
the interesting case where the $\lambda_k$'s for ${k\in S^*}$'s are much
larger than $\lambda_p$, by using the weighted LASSO procedure, we
only pay for ill-posedness in the support of $x^*$, without even
knowing this support. Note also that in this set-up, if one wants to
choose a constant weight $d$, then
$d\simeq \max_k{(\sigma_k/\lambda_k)}$ and one again pays for global
ill-posedness and not just ill-posedness in the support of $x^*$.

This toy example shows us three things:
\begin{enumerate}
\item[(i)] The $d_k$'s are indeed balancing both ill-posedness and
  heteroscedasticity of the problem.
\item[(ii)] The choice of the mappings from $A$ to $\tA$ and $Y$ to
  $\tY$ impacts the rates.
\item[(iii)] The non-constant $d_k$'s allow for ``adaptivity'' with
  respect to the local ill-posedness of the problem, in terms of the
  support of $x^*$.
\end{enumerate}

\revision{Of course, this example is just a toy example and many
  simplifications occur due to the diagonalization effect; however,
  the same phenomena appear in the much more intricate examples
  (Bernoulli and Convolution) of Sections \ref{sec:bernoulli} and
  \ref{sec:conv}. To deal with these settings, we need to choose
  ${\tA}$ so that the corresponding Gram matrix
  $\tG={\tA}^\T{\tA}$ is as homogeneous as possible (meaning
  that there is no great discrepancy between maximal and minimal
  eigenvalues on restricted sets typically and informally that it
  looks as much as possible as the identity matrix up to a
  multiplicative constant) and choose ${\tY}$ such that
  ${\tA}^\T({\tY}-{\tA}x^*)$ is as small as possible, which will make
  the $d_k$'s as small as possible and therefore giving the best
  possible rates, that could not be achieved using a single constant
  weight $d$.}  This choice in particular will enable us to get rates
consistent with the minimax rates derived in \cite{PCS_garvesh} in a
slightly different framework.

\subsection{Organization of the paper}
Section~\ref{sec:bounds} describes general oracle inequalities,
recovery rate guarantees, and support recovery bounds for the three
estimators described above, given weights which satisfy
\eqref{eq:condition}. We then describe a general framework for finding
such weights using the observed data in Section~\ref{sec:weights}. We
next describe exact weights and resulting risk bounds for two specific 
Poisson inverse problems: (a) Poisson compressed sensing using a
Bernoulli sensing matrix, which models certain optical imaging systems
such as \cite{riceCamera}, and (b) a ill-posed Poisson deconvolution
problem arising in genetic motif analysis, building upon the
formulation described in \cite{laure}. We conclude with
simulation-based verification of our derived rates.

\subsection{Notation}
To provide readable results, we use the following notation in the
sequel: $a \lesssim_\gamma b$ if there exists a positive constant
$c_\gamma$ only depending on $\gamma$ such that $a \leq c_\gamma b$.
Similarly, $a \gtrsim_\gamma b$ means $b \lesssim_\gamma a$ and
$a \simeq_\gamma b$ means both $a \lesssim_\gamma b$ and
$a \gtrsim_\gamma b$. If there is no index $\gamma$, it just means
that the constants are absolute. In the proofs, the notation $\square$
represents an absolute constant that may change from line to line.

\section{Theoretical performance bounds for the weighted LASSO}
\label{sec:bounds}

In this section, we establish recovery error bounds for the
  proposed weighted LASSO estimator. The underlying proof techniques
  closely follow those described in
  \cite{negahban2009unified,hastie2015statistical} and elsewhere, but
  have been adapted to account for the weighted-$\ell_1$
  regularizer. Without this adaptation, directly applying the theory
  of \cite{negahban2009unified} to the weighted LASSO estimator yields
  rates that are equivalent to those of the standard LASSO estimator,
  modulo a constant factor. As shown below, our modified analysis
  yields tighter bounds that better reflects the role of the weights,
  as discussed in detail in the examples in the following sections.
  We state our bounds in this section and the associated proofs in the
  appendix for completeness and clarity. The bounds in this section do
  not depend on the noise distribution and can be used regardless of
  the underlying noise.  They rely upon the following two main
assumptions, both of which are proved to be met with high probability
in two key examples in the next sections.

The first assumption is known as the {\em Restricted Eigenvalue
  Condition} (see
\cite{negahban2009unified,raskutti2010restricted,li2016minimax}):
\begin{pHyp}{RE}{\kb,\ka}\label{as:RE}
There exist $\ka,\kb > 0$ such that 
\begin{equation}
\|\tA x\|_2 \ge \ka \|x\|_2 - \kb \|x\|_1 \qquad \forall \,
x \in \reals^p.
\label{eq:rec}  
\end{equation}
\end{pHyp}
This condition is weaker than the verifiable condition in
\cite{Juditsky}.  Our other key assumption dictates the
necessary relationship between the weights used to regularize the
estimates $\xW$ and $\xL$.
\begin{pHyp}{Weights}{\{d_k\}_k}\label{as:dk}
For $k = 1,\ldots,p$,
\begin{equation}
 |\tA ^\T (\tY-\tA x^*)|_k \leq d_k.  \label{controledk}
\end{equation}
\end{pHyp}

In the sequel, we use the following definitions:
\begin{align}
\dmax \deq \max_{k \in \{1,\dots,p\}}d_k, \qquad \dmin \deq
  \min_{k \in \{1,\ldots,p\}} d_k, \quad  \mbox{and} \quad 
\rho_{\gamma} := \gamma\frac{\gamma+2}{\gamma-2}. 
\label{eq:rho}
\end{align}
For any vector $z \in \reals^p$ and
  $S \subseteq \{1,\ldots,p\}$, let $z_S \in \reals^p$ be defined via
$(z_S)_i = \begin{cases} z_i, & i \in S\\
0,& \mbox{otherwise}
\end{cases}$.
 Further recall that $D$ is a diagonal matrix with the $k^{\rm th}$ diagonal element
equal to $d_k$.
Because $D$ is diagonal, note that for any vector $z\in\R^p$ and  any
set $S \subseteq \{1,\ldots,p\}$, $Dz_S = (Dz)_S$.

\begin{proposition}\label{th:L2}
  Fix $\eps > 0$. If $\gamma>2$ and
  Assumptions~\asref{as:dk}{\{d_k\}_k} and~\asref{as:RE}{\kb,\ka} are
  satisfied, then there exists a universal constant $c>0$ such that
  for any set $S \subseteq \{1,\ldots,p\}$ for which
\begin{equation}\label{mainc}
\|d_S\|_2 \le \dmin \frac{\ka-\eps}{\kb \rho_\gamma},
\end{equation}
for $\eps>0$, the weighted LASSO estimator satisfies
\begin{equation}\label{mainr}
\| x^* - \xW \|_2 \leq c\left(\frac{\rho_{\gamma}}{\eps^2}\|d_S\|_2 + \sqrt{\frac{\rho_{\gamma}}{\eps^2}}\|Dx^*_{S^c}\|_1^{1/2}
  +  \frac{\rho_{\gamma}\kb}{\eps\dmin}\|Dx^*_{S^c}\|_1\right)
\end{equation}
Furthermore, for any set $S \subseteq \{1,\ldots,p\}$ with $s = |S|$ satisfying
\begin{equation}\sqrt{s} \le \frac{\ka - \eps}{\kb \rho_\gamma},
\label{mainc2}
\end{equation}
for $\eps>0$, the LASSO estimator satisfies
\begin{align}\label{mainrbis}
\| x^* - \xL \|_2 \le &c\left(\frac{\rho_{\gamma}}{\eps^2} d \sqrt{s} + \sqrt{\frac{\rho_{\gamma} d}{\eps^2}}\|x^*_{S^c}\|_1^{1/2}
  +  \frac{\rho_{\gamma}\kb}{\eps}\|x^*_{S^c}\|_1\right).
\end{align}
\end{proposition}

Note that Inequality \eqref{mainr} can be expressed simply in the
\revision{case where the vector $x^*$ is $s$-sparse with support $S^*$ where $s=|S^*|$ and $S^*$ satisfies \eqref{mainc}:}
\begin{equation}
\| x^* - \xW \|_2 \leq c\frac{\rho_{\gamma}}{\eps^2}\|d_{S^*}\|_2.
\label{eq:sparseRate}
\end{equation}
This result clearly shows the importance of having weights as small as
possible but large enough so that Assumption~(\ref{as:dk}) is
satisfied and not too heterogeneous so that \eqref{mainc} is
true. This trade-off for weights choice will be illustrated in
Sections \ref{sec:bernoulli} and \ref{sec:conv}.  Note that 
\eqref{mainc} is equivalent to 
$$\sqrt{s} \le \frac{\|d_S\|_2}{\dmin} \le \frac{\ka - \eps}{\kb
  \rho_{\gamma}}$$
with $s = |S|$, so that this condition is both a sparsity condition
and a limit on the heterogeneity of the weights.

\section{Choosing data-dependent weights}
\label{sec:weights}
In general, choosing $d_k$'s to ensure that 
Assumption~\asref{as:dk}{\{d_k\}_k} is satisfied is highly
problem-dependent, and we give two explicit examples in the following
two sections. In this section we present the general strategy we adopt
for choosing the weights. The weights $d_k$ are ideally chosen so that for all $k$
\begin{equation}
|(\tA^\T (\tY - \tA x^*))_k|\leq d_k.
\label{eq:unbiased}
\end{equation}
We describe a data-dependent strategy for choosing weights such that
this condition holds with high probability.  The modifications $\tY$
and $\tA$ of $Y$ and $A$ that we have in mind are linear, therefore
one can generally rewrite for each $k$,
$$(\tA^\T (\tY - \tA x^*))_k = R_k^\T (Y-Ax^*) + r_k(A,x^*),$$ for
some vector $R_k \in \reals^n$ which depends on $k$ and $A$, and for some
residual term $r_k(A,x^*)$, also depending on $k$ and $A$. The
transformations are chosen such that $d_k$ is small. With the above
decomposition, the first term $R_k^\T (Y-Ax^*)$ is naturally of null
conditional expectation given $A$ and therefore of zero mean. The
$\tY$ are usually chosen such that $r_k(A,x^*)$ is also of zero mean
which globally guarantees that $\E[(\tA^\T (\tY - \tA x^*))_k]=0$.

In the two following examples, the term $r_k(A,x^*)$ is either mainly
negligible with respect to $ R_k^\T (Y-Ax^*)$ (Bernoulli case,
Section~\ref{sec:bernoulli}) or even identically zero (convolution
case, Section~\ref{sec:conv}).  Therefore the weights are mainly given
by concentration formulas on quantities of the form $R^\T (Y-Ax^*)$
as given by the following Lemma.

\begin{lemma}
\label{concPoisGen}
For all vectors $R=(R_\ell)_{\ell=1,\ldots,\nObs} \in \reals^\nObs$, eventually depending on $A$, let $R_2:=(R_\ell^2)_{\ell=1,\ldots,\nObs}$. Then
the following inequality holds for all $\theta>0$,
\begin{equation}\label{unC}
\P\left( R^\T Y\geq  R^\T A x^*+ \sqrt{2v\theta}+\frac{b\theta}{3}\,\Big|\,
  A\right) \leq e^{-\theta},
  \end{equation} 
with
$$v = R_2^\T \E(Y|A)=  R_2^\T A x^*$$
and
$$b =\norm{R}_\infty.$$
Moreover
\begin{equation}\label{deuxC}\P\left( |R^\T Y -R^\T A x^*|\geq\sqrt{2v\theta}+\frac{b\theta}{3}\,\Big|\,
  A\right) \leq 2 e^{-\theta},
 \end{equation}  
\begin{equation}\label{troisC} \P\left( v \geq \left(\sqrt{\frac{b^2 \theta }{2}}+\sqrt{\frac{5b^2\theta}{6}+ R_2^\T Y }\right)^2\,\Big|\,
  A\right) \leq  e^{-\theta}
  \end{equation}
  and
\begin{equation}\label{quatreC}\P\left( |R^\T Y -R^\T A x^*|\geq \left(\sqrt{\frac{b^2 \theta }{2}}+\sqrt{\frac{5b^2\theta}{6}+ R_2^\T Y }\right) \sqrt{2 \theta}+\frac{b\theta}{3}\,\Big|\,
  A\right) \leq 3 e^{-\theta}.
 \end{equation}
\end{lemma} 

Equations \eqref{unC} and \eqref{deuxC} give the main order of
magnitude for $R^\T (Y-Ax^*)$ with high probability but are not
sufficient for our purpose since $v$ still depends on the unknown
$x^*$. That is why Equation \eqref{troisC} provides an estimated
upper-bound for $v$ with high probability. Equation \eqref{quatreC} is
therefore our main ingredient for giving observable $d_k$'s that
satisfy Assumption~\asref{as:dk}{\{d_k\}_k}. Note that, depending on
$A$, one may also find more particular way to define those weights, in
particular constant ones. This is illustrated in the two following
examples.

\section{\revision{Case Study}: Photon-limited compressive imaging}
\label{sec:bernoulli}
A widely-studied compressed sensing measurement matrix is the
Bernoulli or Rademacher ensemble, in which each element of $A$ is
drawn iid from a Bernoulli($q$) distribution for some $q \in
(0,1)$.
(Typically, $q=1/2$.) In fact, the celebrated Rice
single-pixel camera \cite{riceCamera} uses exactly this model to
position the micromirror array for each projective measurement. This
sensing matrix model has also been studied in previous work on Poisson
compressed sensing ({\em cf.} \cite{PCS_garvesh,pcs}). In this
section, we consider our proposed weighted LASSO estimator for this
sensing matrix. \revision{Because our focus is a comparison of the
  classical and weighted LASSO estimators, we focus here on $s$-sparse
  $x^*$, which is consistent with previous theoretical analyses of
  this problem. Note, however, that our results extend trivially to
  the non-sparse setting.}
  
\subsection{Rescaling and recentering}
Our first task is to define the surrogate design matrix $\tA$ and
surrogate observations $\tY$.  In this set-up, one can easily see that
the matrix
\begin{equation}\label{tildeA}
{\tA}= \frac{A}{\sqrt{n q(1-q)}}-\frac{q  \ones_{\nObs \times 1} \ones_{p \times
  1}^\T}{\sqrt{nq(1-q)}}
\end{equation}
is a scaled and shifted version of the original $A$ and satisfies $\E({\tA}^\T{\tA})=I_p$ (see Appendix), which will help us
to ensure that \revision{Assumption \ref{as:RE} holds}.  To make $d_k$
as small as possible while still satisfying Assumption
\asref{as:dk}{\{d_k\}_k}, we would like to have
$\E({\tA}^\T({\tY}-{\tA}x^*))=0$, as stated previously. Computations given in the appendix show that it is sufficient to take
  \begin{equation}\label{tildeY}
  {\tY}= \frac{1}{(n-1)\sqrt{nq(1-q)}} (nY-\sum_{\ell=1}^\nObs Y_\ell  \ones_{\nObs \times 1}).
  \end{equation}

\subsection{Assumption~\ref{as:RE} holds with high probability.}
\begin{proposition}\label{REnewCS}
There exist positive absolute constants $c^{\prime}$ and $c^{\prime \prime}$ such that with probability larger than $1-c^{\prime}\exp(-c^{\prime \prime}n)$, Assumption~$\ref{as:RE}({\kb,\ka})$ holds with
$$\kb=\frac{c}{q(1-q)}\sqrt{\frac{\log p}{n}}\quad\mbox{and}\quad\ka=\frac{1}{4},$$
where $c$ is an absolute positive constant.
\end{proposition}

In the sequel, we focus on the interesting case when $q$ is small and tends to $0$ as $n$ and $p$ grow. Therefore we assume that
\begin{equation}\label{condq}
0<q\leq \frac{1}{2}\leq 1-q<1.
\end{equation}
In addition to $q$ being small, the orders of magnitude  that we have derived  only hold for
\begin{equation}\label{range-ber}
q\gtrsim \sqrt{\frac{\log(p)}{n}},
\end{equation}
which implies in particular that
\begin{equation}\label{range-pn}
\log(p)\lesssim n.
\end{equation}
In particular, $q$ can still tend to $0$ with $n$ and $p$ but cannot be too small, as long as 
$
\log(p)\ll n.$

\subsection{Choice of the weights}
We now discuss the rates obtained by applying Proposition \ref{th:L2} for constant and non-constant weights that are found thanks to the machinery described in Section~\ref{sec:weights}.

\subsubsection{Definition of constant weight and rates for the estimate $\xL$}
For all  $k=1,\ldots,p$,
define
the vector $V_k \in \reals^n$ so that the $\ell^{\rm th}$ element is
\begin{equation}
V_{k,\ell}:=\left(\frac{na_{\ell,k}-\sum_{\ell'=1}^\nObs
    a_{\ell',k}}{n(n-1)q(1-q)}\right)^2
\label{eq:Vk}
\end{equation}
and let $A_k \in \reals^n$ denote the $k^{\rm th}$ column of $A$. 

Let $$W=\max_{u,k\in\{1,\ldots,p\}} \ave{A_u,V_k} $$ and
$$\hat{N}=\frac{1}{nq-\sqrt{6nq(1-q)\log(p)}-\max(q,1-q)\log(p)} \left(\sqrt{\frac{3\log(p) }{2}}+\sqrt{\frac{5\log(p)}{2}+ \sum_{\ell=1}^n Y_\ell }\right)^2
$$
be an estimator of $\norm{x^*}_1$.  Then one can choose constant
weight defined by
\begin{equation}
d =  \sqrt{\hat{N} 6W\log(p)}   +
\frac{\log(p)}{(n-1)q(1-q)}+ c \left(\frac{3\log(p)}{n} +
  \frac{9\max(q^2,(1-q)^2)}{n^2q(1-q)}\log(p)^2\right)  \hat{N},
\label{eq:berd}
\end{equation}
where $c$ is an absolute constant (see the proof of
Proposition~\ref{dBern}, $c=126$ works if $n\geq 20$). One can prove
that this satisfies Assumption \asref{as:dk}{d} except on an event of
probability of order $1/p$ as long as $p\geq 2$ (see
Proposition~\ref{dBern}).

If we want to give an explicit rate for the corresponding classical LASSO estimator, we need to find the order of magnitude of $d$.
As  shown in
Proposition~\ref{bounddBern}, in the range \eqref{range-ber},
we have
$$ d \simeq \sqrt{\frac{\log(p) \norm{x^*}_1}{nq}} +
  \frac{\log(p)\norm{x^*}_1}{n}+
  \frac{\log(p)}{nq}.$$

We now apply Proposition \ref{th:L2} with $\varepsilon=1/8$ and  a fixed $\gamma>2$. It is easy to see that if the size of the true support $s $ satisfies
\begin{equation}\label{sparse-ber-L}
s  \ll \frac{nq^2}{\log p},
\end{equation}
then \eqref{mainc} holds. Hence, all assumptions of Proposition~\ref{th:L2} are satisfied and 
we have that, except on an event of probability of order $1/p+e^{-c"n}$,
\begin{equation}
\| \xL-x^*\|_2^2 \lesssim_\gamma
\frac{\log p}{n}\left(
\frac{\|x^*\|_1 s }{q} +  
\frac{\|x^*\|^2_1 s  \log p}{n} + \frac{s  \log p}{n q^2}
\right).
\label{eq:cslasso}
\end{equation}
In particular, in the range
\begin{equation}\label{rangex}
1\lesssim \|x^*\|_1 \lesssim n/\log(p),
\end{equation}
\revision{the first term of \eqref{eq:cslasso} dominates and we have
  $\| \xL-x^*\|_2^2 \lesssim_\gamma \frac{s  \|x^*\|_1 \log p }{nq}$.}

\subsubsection{Definition of non-constant weights and rates for the estimate $\xW$}
One can choose the non-constant weights defined by
\begin{align} d_k =   \sqrt{6\log(p)}  \left(\sqrt{\frac{3\log(p)
  }{2(n-1)^2q^2(1-q)^2}}+\sqrt{\frac{5\log(p)}{2(n-1)^2q^2(1-q)^2}+
  \ave{V_k, Y} }\right) \nonumber \\+\frac{\log(p)}{(n-1)q(1-q)}+c
  \left(\frac{3\log(p)}{n} +
  \frac{9\max(q^2,(1-q)^2)}{n^2q(1-q)}\log(p)^2\right)  \hat{N},
\label{eq:berdk}
\end{align}
where $V_k$ is defined in \eqref{eq:Vk} and $c$ is an absolute
constant (see the proof of Proposition~\ref{dkBern}, $c=126$ works if
$n\geq 20$). They also satisfy Assumption~\asref{as:dk}{\hd} except on
an event of probability of order $1/p$ as long as $p\geq 2$ (see
Proposition~\ref{dkBern}).  Furthermore, as shown in
Proposition~\ref{bounddkBern}, in the range \eqref{range-ber}, we have
the following order of magnitude
\begin{align*} 
d_k \simeq \sqrt{\log(p)\left[\frac{x^*_k}{nq}+\frac{\sum_{u\not = k} x^*_u}{n}\right]} +  \frac{\log(p)\norm{x^*}_1}{n}+   \frac{\log(p)}{nq}.
\end{align*}
For $S^*$ the support of $x^*$, note that
$$\frac{\|d_{S^*}\|_2}{d_{\min}} \lesssim \sqrt{s  +\frac{1}{q}}.$$
Therefore \eqref{mainc} is satisfied as soon as 
\begin{equation}\label{sparse-ber-W}
s  \ll \frac{nq^2}{\log p} \mbox{ and } q \gg \left(\frac{\log(p)}{n}\right)^{1/3}.
\end{equation}
The first part is exactly \eqref{sparse-ber-L}, and the second part is
slightly stronger than \eqref{range-ber}. However, $q$ can still tends
to $0$ with $p$ and $n$ as long as $\log(p)\ll n$. 

Under \eqref{sparse-ber-W}, we can now apply Proposition \ref{th:L2}
as before: except on an event of probability of order
$1/p+e^{-c^{\prime \prime}n}$,
\begin{equation}\| x^*-\xW\|_2^2\lesssim_\gamma
\frac{\log p}{n}
\left(  
\frac{\|x^*\|_1}{q} + \|x^*\|_1 s + 
\frac{\|x^*\|^2_1 s  \log p}{n} + \frac{s  \log p}{n q^2 }
\right).
\label{eq:CSWL}
\end{equation}
In particular, in the range 
$1\lesssim \|x^*\|_1 \lesssim n/\log(p)$  and \eqref{sparse-ber-W}, 
\revision{the first two terms of \eqref{eq:CSWL} dominates and we have
  $\| \xW-x^*\|_2^2 \lesssim_\gamma \frac{\log p \|x^*\|_1
    (s +1/q)}{n}$.
Therefore, one can form the ratio of the upper bounds derived in the
classical setting and the one obtained here, leading to
$$ 
\frac{\frac{\log p}{n}
\left( 
\frac{\|x^*\|_1}{q} + \|x^*\|_1 s \right)}{\frac{\log p}{n}
\left(\|x^*\|_1 s /q \right)}  \lesssim \frac{1}{s }+q,$$
which is much smaller than $1$. The upper bound on the weighted LASSO
estimator is therefore better in this range. }

\subsection{Comparison with the oracle least squares estimate}
\label{sec:CSOLS}
 Let  $S^* \deq \supp(x^*)$ denote the true signal support; we consider
  the {\it oracle least squares estimate} on~$S^*$:
\begin{align}
\xls \deq & I_{S^*} (\tA _{S^*})^\# \tY
\label{eq:LS}
\end{align}
where $\tA _{S^*} \in \reals^{n \times s}$ is a submatrix of $\tA $
with columns of $\tA _{S^*}$ equal to the columns of $\tA $ on support
set $S^*$, $(\tA _{S^*})^\#$ standing for its pseudo-inverse and
$I_{S^*} \in \reals^{p \times s}$ is a submatrix of the identity
matrix $I_p$ with columns of $I_{S^*}$ equal to the columns of $I_p$
on support set $S^*$.  Note that this estimator functions as an
oracle, since in general the support set $S^*$ is unknown. The oracle
least squares estimate estimate satisfies the following result:
\begin{proposition}\label{prop:OLS}
  Assume that $S^* = \supp(x^*)$ is known and $s  = |S^*|$.  Under \eqref{sparse-ber-L},  $\xls$, the least
  square estimate of $x^*$ on $S^*$ satisfies
$$\sum_{k \in S^*} (\tA^\T (\tY-\tA
x^*))_k^2\lesssim \| \xls-x^*\|_2^2 \lesssim \sum_{k \in S^*} (\tA^\T
(\tY-\tA x^*))_k^2$$
 with probability exceeding
$1-C/p$ for a universal positive constant $C$.
\end{proposition}
It's quite natural to compare our LASSO estimates with the OLS
estimate. 
In the same spirit
as \cite{candes2007dantzig}, the OLS estimate can then be viewed as a
benchmark and, in some sense, the previous result shows the optimality
of the weighted LASSO estimator $\xW$.
Specifically, recall from Proposition~\ref{th:L2} and \eqref{eq:sparseRate}, in the sparse setting
$\| x^* - \xW \|_2^2 \lesssim_\gamma \sum_{k \in S^*} d_k^2$ and $\|
x^* - \xL \|_2^2 \lesssim_\gamma sd^2$. By choosing each $d_k$ to be
as close as possible to $(\tA^\T
(\tY-\tA x^*))_k$, we ensure that the risk of the weighted LASSO
estimator $\xW$  is as close as possible to that of the OLS
estimate. In contrast, since $d$ is a bound on $\max_k (\tA^\T
(\tY-\tA x^*))_k$, the classical LASSO estimator $\xL$ may exhibit
higher errors than the OLS.

\subsection{Comparison with previous rate results for Poisson CS}
\revision{We
note that the rates above are similar in spirit to the rates derived in a
similar setting for estimators based on minimizing a regularized
negative log-likelihood. Specifically, we compare our rates to those
derived in \cite{PCS_garvesh}. In that work, the authors consider
observations 
$$Z \sim {\cal P}(T B D \theta^*)$$
where $T>0$ is a scalar parameter controlling
the expected total number of photons collected, $B$ is a normalized, non-negative
sensing matrix reflecting the physical constraints of optical systems,
$D$ is an orthonormal sparsifying basis, and $\theta^*$ is an $s $-sparse
set of basis coefficients. \cite{PCS_garvesh} focused on the case
where $\|x^*\|_1 = \|D \theta^*\|_1 = 1$ so that $T$ alone reflected
the overall signal strength, and where one of the basis vectors in $D$
corresponded to a constant vector. The authors noted that the squared
error of their penalized likelihood estimator scaled like $\frac{s 
  \log p}{T}$ for $T$ sufficiently large. }

\revision{To compare that work with ours, we assume that the results
  in \cite{PCS_garvesh} generalize to the case where the sparsifying
  basis $D$ is the identity matrix. We also focus on the case where
$q=1/2$ and
  the sensing matrix $B$ is generated as follows:\footnote{In
    \cite{PCS_garvesh}, $B_{i,j} = \frac{\tA_{i,j}}{4\sqrt{n}} + \frac{3}{4n}$; that variation ensures Poisson
    intensities are bounded away from zero to facilitate analysis of
    the Poisson log-likelihood. We assume the work of
    \cite{PCS_garvesh} generalizes to the setting described in the text.}
$$B_{i,j} = \frac{\tA_{i,j}}{2\sqrt{n}} + \frac{1}{2n}, \qquad i =
1,\ldots,n;\, j = 1,\ldots,p.$$ 
From here we can see that if $T=n$, then $TB = A$. As a result, the
rate of $\frac{s  \log p}{T}$ in \cite{PCS_garvesh} is on the same
order as the bound \eqref{eq:CSWL},
$\frac{\|x^*\|_1(\frac{1}{q}+s )\log p}{n}$, when $\|x^*\|_1 = 1$, $q=1/2$, and
$n \gtrsim \log p$. 
 }

\section{Case study: Poisson random convolution in genomics}
\label{sec:conv}

This set-up is much less widely studied than the previous one, but it
was the motivating problem at the origin of the present
article. Indeed, this specific random convolution model is a toy model
for bivariate Hawkes models or more precise Poissonian interaction
functions \cite{HRBR12,laure,STM}. Those point processes models have
been used in neuroscience (spike train analysis) to model excitation
from on neuron on another one or in genomics to model distance
interaction along the DNA between motifs or occurrences of any kind of
Transcription Regulatory Elements (TRE) \cite{GS,CSWH}. All the
methods proposed in those articles assume that there is a finite
``horizon'' after which no interaction is possible (i.e. the support
of the interaction function is finite and much smaller that the total
length of the data) and so the corresponding inverse problem is
well-posed.  However, and in particular in genomics, it is not at all
clear that such a horizon exists. Indeed, it is usually assumed that
the interaction stops after 10000 bases. However, the 3D structure of
DNA makes long-range ``linear distances'' on the DNA strand
potentially irrelevant. An important question receiving increased
attention is whether, if one had access to real 3D positions (and
there is ongoing work to measure these positions), would it be
possible to estimate the interaction functions without any assumption
on its support?

The problem described here is a clear simplification of this complex
problem, which in fact depends on the DNA fold: we restrict ourselves
to the case where the DNA strand is just modeled by a circle (which is
topologically reasonable for certain genomes) and the observations of
TREs are binned. We wish to understand whether long range dependencies
can be recovered once sparsity is assumed.  Specifically, we formulate
this problem as a sparse Poisson deconvolution problem.

Our model considers the interdependencies between two different
kind of locations; for example, the first might be a transcription factor binding
site (TFBS) and the second might be a transcription start site
(TSS). Borrowing from the point process terminology, we call the first
set of  occurrences ``parents'' and the second set of 
occurrences ``children''.  We assume the locations of the
parents  along the genome follow a uniform
distribution, and each parent  independently generates children
 in the surrounding genome according to the same distribution
centered around the parent. This idea is illustrated in Figure~\ref{fig:motifs}.
\begin{figure}[!ht]  
\centering
\subfloat[Sparse parent TREs]{\includegraphics[width=.65\textwidth]{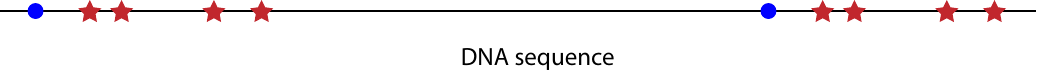}}\\
\subfloat[Dense parent TREs]{\includegraphics[width=.65\textwidth]{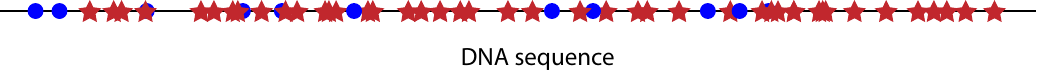}}
\caption{Illustration of the parents/children model. The locations of parents are
  shown as blue dots along the DNA sequence, and the locations of  children 
  generated by each parent are shown as red stars following the parent. The distribution of children is the same after each parent,
but the parentage of each child becomes less clear as the density
of parents increases.}
\label{fig:motifs}
\end{figure}

More formally, let $U_1,\ldots,U_m$ be a collection of $m$
i.i.d.\,realizations of a uniform random variable on the set
$\{0,\ldots,p-1\}$; these corresponds to the parents' locations. Each parent $U_i$  give birth independently to
some children. The number of children of $U_i$ at distance $j$ is given by $N_{U_i+j}^i$, which is a Poisson
variable with distribution $\mathcal{P}(x^*_j)$, where $x^*_j$ represents  the likelihood of a child's existence at distance $j$ from its parent.
 Here we
understand $U_i+j$ in a cyclic way, i.e.\,this is actually $U_i+j$
modulus $p$.  We observe at each position $k$ between $0$ and $p-1$
the total number of children regardless of who their parent might be,
{\em i.e.,} 
$$Y_k=\sum_{i=1}^{m} N_k^i.$$ 
This problem can be translated as
follows: we observe the $U_i$'s and the $Y_k$'s, whose conditional distribution 
given the $U_i$'s is 
$$ Y_k\sim \mathcal{P}\left(\sum_{i=1}^m x^*_{k-U_i}\right).$$ 

Given the $U_i$'s, the elements of $Y=(Y_0,\ldots,Y_{p-1})^\T$ are independent.  The aim
is to recover the vector $x^*=(x^*_0,\ldots,x^*_{p-1})^\T$. We assume throughout that $x^*$ is $s$-sparse for
some $1 \le s < \min(m,p)$. The sparsity is a reasonable assumption in
genomics because linked parents and children in our model correspond
to distinct chemical reactions in the underlying biochemical
system.

The above model actually amounts to a random convolution (we are
convolving the signal $x^*$ by the random empirical measure of the
parents). To the best of our knowledge, the analysis of such a convolution
problem is entirely new.  Other authors have studied random
convolution, notably \cite{RombergToeplitz,CCA,CAKE,CandesPlan}, but
those analyses do not extend to the problem considered here. For
example, Cand\`{e}s and Plan \cite[p5]{CandesPlan} consider random
convolutions in which they observe a random subset of elements of the
product $Ax^*$. They note that $A$ is an isometry if the Fourier
components of any row of $A$ have coefficients
with the same magnitude.  In contrast, in our
setting the Fourier coefficients of $A$ are random and do not have
uniform magnitude, and so the analysis in \cite{CandesPlan} cannot be
directly applied in our setting. In particular, the ratio of $p$ (the
number of elements in $x^*$ and the number of measurements) to $\nU$
(the number of uniformly distributed parents) will play a crucial
role in our analysis but is not explored in the existing literature.

\subsection{Poisson random convolution model}

Let us introduce the multinomial variable $\Nb$, defined for all
$k\in \mathbb{Z}$ by
\begin{equation}
\label{Ncount}
\Nb(k)=\mbox{card}\left\{i: \ U_i=k[p]\right\}
\end{equation}
where $k[p]$ denotes $k$ modulo $p$.
It represents the number of parents at position $k$ on the
circle. Note that $\sum_{u=0}^{p-1} \Nb(u)=\nU$, fact which will be extensively used in proofs, and let us denote
$$A=\begin{pmatrix}
\Nb(0) & \Nb(p-1) & \cdots & \Nb(1)\\
\Nb(1)& \Nb(0) & \ddots &\vdots \\
\vdots & \ddots&\ddots & \Nb(p-1)\\
\Nb(p-1) & \cdots & \Nb(1) & \Nb(0)\\
\end{pmatrix};
$$
that is, $A_{\ell,k} = \Nb(\ell-k)$.
Using this notation, we have in fact the observation model
$$Y \sim \mathcal{P}(A x^*).$$
Note that 
$\E(A)=m\ones_p\ones_p^\T$. Therefore, in expectation, all its
eigenvalues are null except the first one. In this sense it is a badly
ill-posed problem despite the fact that the sensing matrix is square
here. The ill-posedness can also be viewed by the fact that in
expectation, we are convolving the unknown $x^*$ with a uniform
distribution, which is known to be an unsolvable problem. Therefore
(as in many works on compressed sensing), we rely on the randomness to
prove that Assumption~\asref{as:RE}{\kb,\ka} is satisfied with high
probability.

Finally, since $m$ corresponds to the number of parents, it controls
the total number of (randomly shifted) copies of $x^*$ that are
observed and hence determines the expected total number of observed
events, which is $m \|x^*\|_1$. In this sense, $m$ controls the
Poisson rates and signal-to-noise ratio, and also controls the
ill-posedness of the inverse problem.

\subsection{Rescaling and recentering}
As for the Bernoulli case, we first rescale  and recenter the sensing matrix 
$$\tA \deq \frac{1}{\sqrt{\nU}}
A-\frac{\sqrt{\nU}-1}{{p}}\ones\ones^\T,$$
which satisfies that $\E(\tA^\T\tA)=I_p$.
Moreover for any $k\in\{0,\ldots,p-1\}$,
we can easily define:
$$\tY_k \deq \frac{1}{\sqrt{\nU}} Y_k - \frac{\sqrt{\nU}-1}{{p}}
\ybar,$$
where $\overline{Y} = \frac{1}{m} \|Y\|_1$. 
Note that because of the particular form of $A$, $\expect[\tY|\tA] = \tA x^*$, (see Lemma \ref{centering} in the Appendix) which explains why in this case the remainder term $r_k$ described in Section \ref{sec:weights} is actually null.

\subsection{Assumption~\ref{as:RE} holds with high probability}
Let $\tG := \tA^\T \tA$.
\begin{proposition}\label{GxiConv}
  There exists absolute positive constants $\kappa$ and $C$ and an
  event of probability larger than $1-C/p$, on which
 \begin{equation}\label{cond_xi}
\forall k, \ell, \left| \left( \tG - I_p \right)_{k,\ell} \right| \le \xi :=\kappa \left( \frac{\log p}{\sqrt{p}} +
    \frac{\log^2 p}{\nU}\right).
\end{equation}
This implies on the same event that Assumption~$\ref{as:RE}({\kb,\ka})$ holds with
$$\kb=\sqrt{\xi}\quad\mbox{and}\quad\ka=1.$$
\end{proposition}
 This result is a result of concentration inequalities for $U$-statistics (see Proposition \ref{MatG} in the appendix).

\subsection{Choice of the weights}
As in the Bernoulli case (Section~\ref{sec:bernoulli}), we consider
both a constant weight (corresponding to a classical LASSO estimator)
and non-constant weights that follows from Lemma \ref{concPoisGen}.

\subsubsection{Definition of constant weight and rates for the estimate $\xL$}
For all $k = 1,\ldots,p$, define the vector $V_k \in \reals^p$ so that
the $\ell^{\rm th}$ element is
\begin{equation}
V_{k,\ell}:= \left( \frac{\mathbb{N}(\ell-k)-\frac{m-1}{p}}{m}\right)^2.
\label{eq:Vk2}
\end{equation}
Let $$W = \max_{k \in\{ 1,\ldots,p\}} \ave{V_k,\mathbb{N}}.$$
Furthermore, let
$$B=\max_{u\in \{0,\ldots,p-1\}} \frac{1}{\nU} \left| \Nb(u) -
  \frac{\nU-1}{p}\right| = \|V_1\|_\infty^{1/2}.$$
Then one can choose the following constant weight 
\begin{equation}\label{def:cst-W}
\hd: = \sqrt{4 W\log p}\left [\sqrt{\ybar+\frac{5\log
      p}{3\nU}}+\sqrt{\frac{\log p}{\nU}}\right ]  + \frac{2B\log p}{3}
\end{equation}
and one can prove that it satisfies Assumption \asref{as:dk}{d} except on an event of probability of order $1/p$ (see Proposition~\ref{ConstConv}). 

The order of magnitude of $\hd$ is given by Proposition~\ref{upboundd} in the appendix and states that
$$\hd^2 \lesssim \left( \frac{\log(p)^2}{p} + \frac{\log(p)^3}{m}\right) \left(\norm{x^*}_1 + \frac{\log(p)}{m}\right).$$
We consider once again a sparse signal $x^*$ with support size $s $. By fixing $\varepsilon=1/2$ and $\gamma>2$, it is easy to see that \eqref{mainc} is implied for large $p$ by
\begin{equation}\label{sparse-conv-L}
s  \ll \min\left(\frac{\sqrt{p}}{\log(p)},\frac{m}{\log(p)^2}\right)
\end{equation}
and applying Proposition \ref{th:L2} gives that except on an event of probability of order $1/p$,
\begin{equation}\label{CL}
\norm{\xL-x^*}_2^2\lesssim_\gamma  s 
\left(\frac{\log(p)^2}{p}+\frac{\log(p)^3}{m}\right)\left(\norm{x^*}_1+\frac{\log(p)}{m}\right). 
\end{equation}

\subsubsection{Definition of non-constant weights and rates for the estimate $\xW$}
For all $k=0,\ldots,p-1$, one can choose the non-constant weights given by
\begin{equation} \label{def:noncst-W}
\hd_k = \sqrt{4 \log p}\left [\sqrt{\ave{V_k,Y}+\frac{5 B^2 \log
      p}{3}}+\sqrt{B^2 \log p}\right ]  + \frac{2B\log p}{3},
 \end{equation}
 with
 for all $k$ in $\{0,\ldots,p-1\}$.
These weights satisfy Assumption~\asref{as:dk}{\hd} except on an event of probability of order $1/p$
(see Proposition~\ref{NonConstConv}). 
The order of magnitude of the $\hd_k$'s are subtle to derive and we have been able to derive them only 
 in the range
\begin{equation}\label{relationshipspourxi}
\sqrt{p} \log p \lesssim m \lesssim \frac{p}{\log(p)}.
\end{equation}
where Proposition~\ref{updownbounddk} shows that
$$ \frac{ x^*_k  \log p}{\nU} +  \frac{\log p}{p} \sum_{u\not = k} x^*_u
  + \frac{\log^2 p}{\nU^2} \lesssim
\hd_k^2 \lesssim \frac{ x^*_k \log p }{\nU} + \frac{\log^2 p}{p} \sum_{u\not = k} x^*_u+ \frac{\log^4 p}{\nU^2}.$$

We consider a sparse signal $x^*$ with support  $S^*$. Then, since $m\lesssim p$
$$\frac{\|d_{S^*}\|_2^2}{\dmin^2}\lesssim
\frac{\|x^*\|_1\left(\frac{\log p}{m}+\frac{s \log^2
      p}{p}\right)+\frac{s \log^4 p}{m^2}}{\frac{\log
    p}{p}\|x^*\|_1+\frac{\log ^2 p}{m^2}} \lesssim \frac{p}{m}+ s
\log^2 p.$$
Since $\xi\simeq \log(p)/\sqrt{p}$ in the range \eqref{relationshipspourxi}, one can then easily see that for fixed $\gamma>2$ and $\varepsilon=1/2$, \eqref{mainc} is implied for large $p$ by
\begin{equation}\label{sparse-conv-W}
s \ll \frac{\sqrt{p}}{\log^3 p}.
\end{equation}
This condition is equivalent up to logarithmic factors to \eqref{sparse-conv-L} which is necessary for the classical LASSO as soon as \eqref{relationshipspourxi} holds.

It remains to apply Proposition \ref{th:L2} to obtain that, under \eqref{relationshipspourxi} and \eqref{sparse-conv-W}, except on an event of probability of order $1/p$, 
\begin{equation}\label{WL}
\norm{\xW-x^*}_2^2\lesssim \left(\frac{  \norm{x^*}_1 \log
    p}{m}+\frac{s \norm{x^*}_1 \log^2 p}{p}+\frac{s \log^4 p}{m^2}\right).
\end{equation}
If the signal is strong enough (i.e., $\norm{x^*}_1 \gg \log(p)/m$)
and in the range \eqref{relationshipspourxi}, the bound on the risk of
classical LASSO estimator, \eqref{CL}, is of order
$\frac{s \norm{x^*}_1 \log^3 p}{m}$. The ratio of the two bounds,
$$\frac{\frac{ \norm{x^*}_1 \log p}{m}+\frac{s \norm{x^*}_1 \log^2p
  }{p}+\frac{s \log^4 p}{m^2}}{\frac{s  \norm{x^*}_1 \log^3 p}{m}}=
\frac{1}{s\log^2 p  }+ \frac{m}{p \log p}+ \frac{\log p}{m\|x^*\|_1},$$
illustrates that the weight LASSO weights are much smaller as $p$
tends to infinity, if $\norm{x^*}_1 \gg \log(p)/m$ and
\eqref{relationshipspourxi} hold.

\revision{\subsection{Comparison with the oracle least squares estimate} 
We consider the {\it oracle least squares estimate} on~  $S^*$, the true signal support:
\begin{align}
\xls \deq & I_{S^*} (\tA _{S^*})^\# \tY;
\label{eq:OLS_C}
\end{align} 
see Section~\ref{sec:CSOLS} for definitions of the notation in this
estimator.  As before, this estimator functions as an oracle, since in
general the support set $S^*$ is unknown. The oracle least squares
estimate estimate satisfies the following result:
\begin{proposition}\label{prop:OLS2}
  Assume that $S^* = \supp(x^*)$ is known and $s  = |S^*|$.  Under \eqref{sparse-conv-L},  then $\xls$, the least
  square estimate of $x^*$ on $S^*$ satisfies
$$  \sum_{k \in S^*} (\tA^\T (\tY-\tA x^*))_k^2 \lesssim \| \xls-x^*\|_2^2\lesssim \sum_{k \in S^*} (\tA^\T (\tY-\tA x^*))_k^2$$
with probability exceeding $1-C/p$ for universal positive constants $C$.
\end{proposition}} See the discussion on the implications of this in
Section~\ref{sec:CSOLS}, which apply in this setting.

\subsection{Simulations}
In this section we simulate the random convolution model described
above and the performance of the (unweighted) LASSO and weighted LASSO
estimators.  We compare the performance of these estimators to an
oracle maximum likelihood estimator (MLE) estimated over the true
support $S^*$ which is denoted by $\wh{x}^{\mbox{MLE}.}$ Note that
this MLE is efficient and therefore should be the estimate with
minimum variance among unbiased estimators and have the minimum risk
in some sense.

We have shown that when $m$ is small relative to $p$, for both
weighted LASSO and least-squares estimators, the MSE upper bound
scales like $\frac{\|x^*\|_1\log p }{m}$;
for the LASSO estimator, the MSE scales like
$\frac{s\|x^*\|_1\log^3 p }{m}$; in the below, we present a simulation
showing that these upper bounds are tight.  \revision{First we examine
  the MSE of the LASSO and weighted LASSO estimators as a function of
  $s$, the sparsity level of $x^*$ for various $m$. We set $p = 5000$
  and $s$ ranges from 10 to 30. The $A$ matrix is randomly generated
  at each experiment. The $x^*$ is fixed for each value of $s$, and
  $\|x^*\|_1$ is kept the same for different $s$ values.  }
\revision{The tuning parameter $\gamma = 2.01$ such that it satisfies
  the constraint $\gamma>2$. Each point in the plots is averaged over
  100 random realizations.}

\revision{ In Figure~\ref{fig:MSEvss}, we explore the MSE as a
  function of $s$ for different values of $m$, with $p=5000$ in the
  Poisson random convolution setting for genomics. For $m$ in the
  range in \eqref{relationshipspourxi}, for fixed $p, m,$ and
  $\|x^*\|_1$, our theory predicts that weighted LASSO outperform
  standard LASSO by a factor of $s$, as reflected in (a)-(d).  Note
  that some of the values of $m$ in these plots are outside the range
  considered in
  \eqref{relationshipspourxi}. Figures~\ref{fig:MSEvss}(a)-(b) satisfy
  the conditions of our theory and behave as expected;
  Figures~\ref{fig:MSEvss}(c)-(d) do not satisfy our conditions and
  hence do not demonstrate the gains predicted by our
  theory. Specifically, as $m$ gets large, we see a greater dependence
  of the MSE on $s$ for the weighted LASSO. This effect is predicted
  by the theory. The weights used by the Weighted LASSO in
  \eqref{def:noncst-W} depend on the variances in $\ave{V_k,Y}$, where
  $V_k$ is defined in \eqref{eq:Vk2}. 
As $m$ grows, the elements of $V_k$ become more uniform as a consequence of the strong law of large numbers, so the
inner products (and hence the $d_k$'s) all start to be close to the same
value. When this happens, the weighted LASSO estimate closely
approximates the classical LASSO estimate, as illustrated in the
figure.  }

\begin{figure}[htp]
\centering
\subfloat[$m=6\times 10^1$]{\includegraphics[width=.32\textwidth]{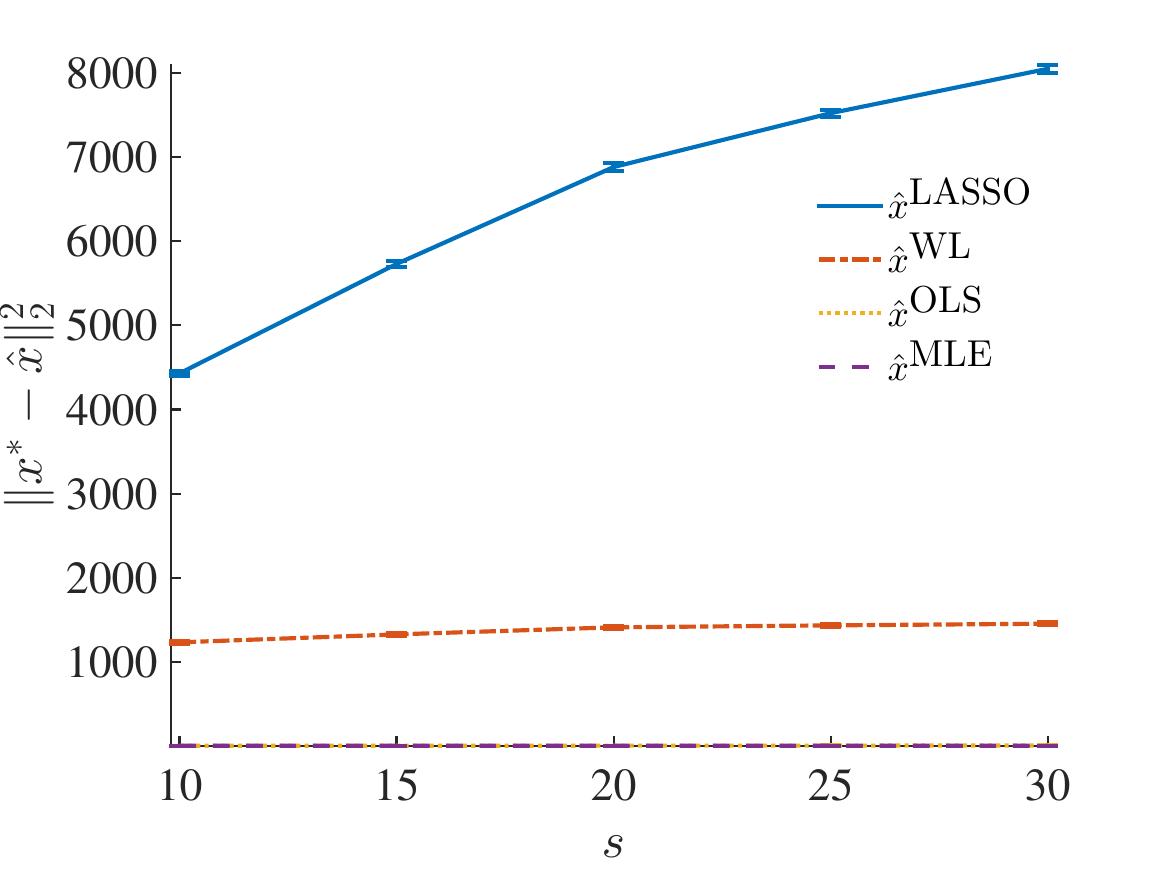}}~
\subfloat[$m=6\times 10^2$]{\includegraphics[width=.32\textwidth]{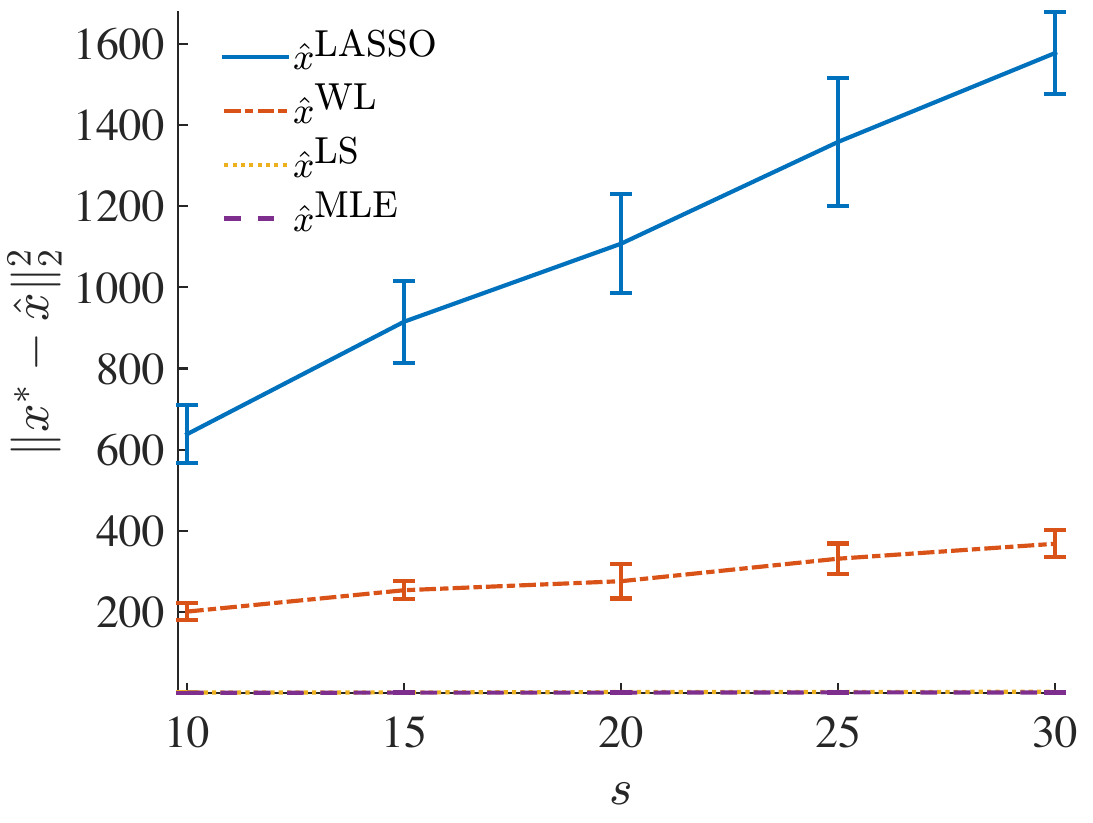}}~
\subfloat[$m=6\times 10^3$]{\includegraphics[width=.32\textwidth]{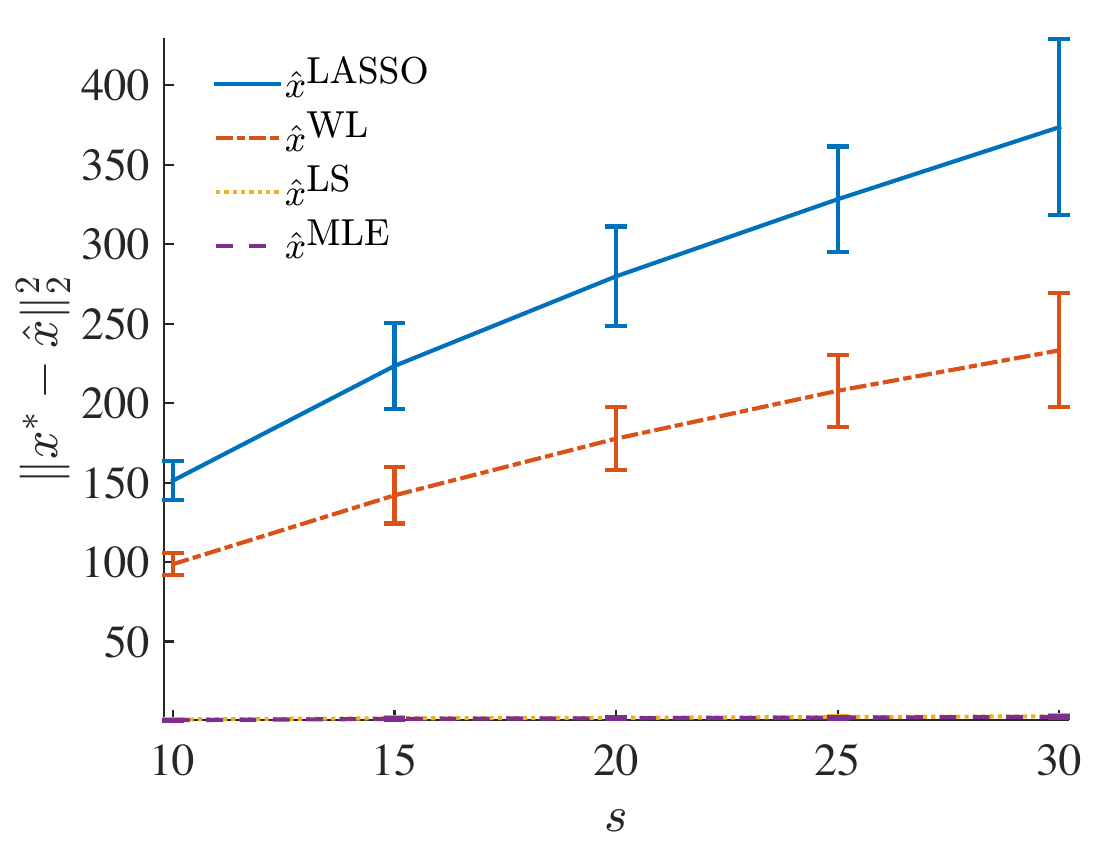}}\\
\subfloat[$m=6\times 10^4$]{\includegraphics[width=.32\textwidth]{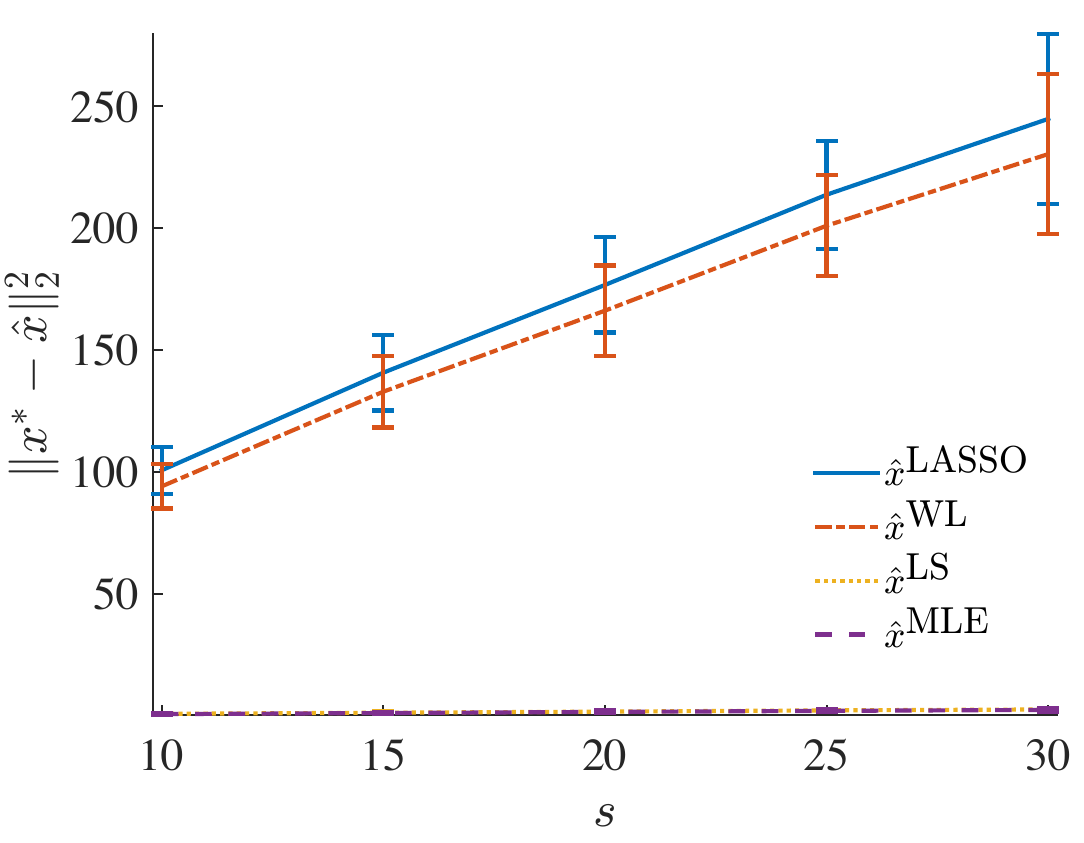}}~
\subfloat[$m=6\times 10^5$]{\includegraphics[width=.32\textwidth]{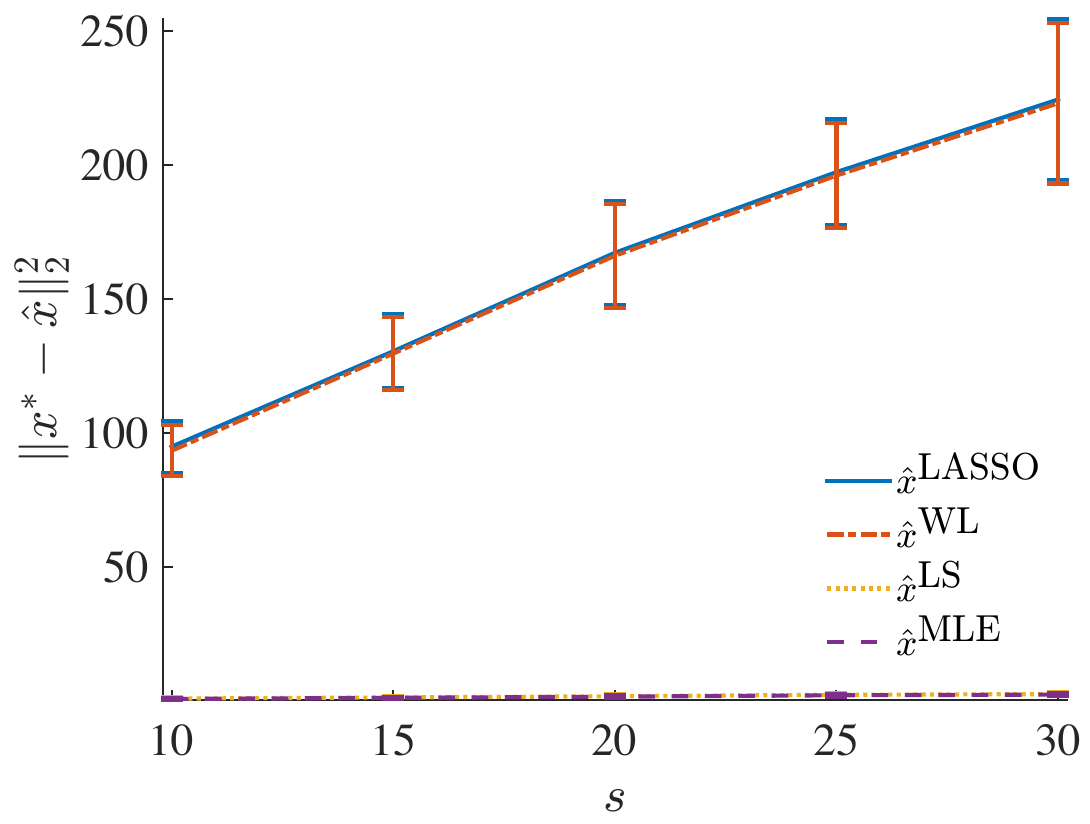}}
\caption{MSE vs $s$ for different values of $m$, with $p=5000$ in the
  Poisson random convolution setting for genomics. This simulation
  compares the estimators oracle least squares \eqref{eq:OLS_C},
  classical (unweighted) LASSO \eqref{eq:unweighted} with $d$ in
  \eqref{def:cst-W}, and the weighted LASSO \eqref{eq:lasso} with
  weights in \eqref{def:noncst-W}.  The oracle maximum likelihood
  estimators yield the smallest estimation error, and the oracle least
  squares estimator yields smaller estimation error than the LASSO and
  weighted LASSO estimators. The weighted LASSO estimator outperforms
  the standard LASSO estimator, as predicted by the theory.  Note that
  some of the values of $m$ in these plots are outside the range
  considered in \eqref{relationshipspourxi}. For $m$ satisfying
  \eqref{relationshipspourxi}, the error of the weighted LASSO does
  not scale with $s$, as predicted by the theory.  As $m$ gets large
  and exceeds this range, however, we see a greater dependence of the
  MSE on $s$ for the weighted LASSO. This effect is predicted by the
  theory. The weights used by the weighted LASSO in
  \eqref{def:noncst-W} depend on the variances in $\ave{V_k,Y}$,
  which become more uniform as $m$ gets large.}
\label{fig:MSEvss}
\end{figure}

\revision{Next, we examine the MSE of the classical and weighted LASSO
  estimators
  as a function of $p$, the length of $x^*$. In this experiment, we
  set $s=10$ and $p$ ranges from 1000 to 10000. For each $p$, we set
  $m\propto \sqrt{p}\log(p)$ ($m$ varies from 11 to 40 in this
  experiment). This specific choice of $m$ is made due to the
  requirement of $m \gtrsim \sqrt{p}\log(p)$ in our rate results
  \eqref{relationshipspourxi}.  The $x^*$ is fixed for each of the $p$
  value, and $\|x^*\|_1$ is kept the same for different $p$. $A$ is
  randomly generated in each trial. The tuning parameter
  $\gamma$ is set to $2.01$ such that it satisfies the constraint
  $\gamma>2$. Each point in the plots is averaged over 100 random
  realizations.  }

\revision{ Figure~\ref{fig:MSEvsp} shows MSE as a functions of
  $p$. The weighted LASSO estimators outperforms the classical LASSO
  estimators. Note that because of our choice of $m$, our theorem
  predicts that the MSE of weighted LASSO estimator scales like
  $\frac{\|x^*\|_1}{\sqrt{p}}$; the MSE of LASSO scales like
  $\frac{s\|x^*\|_1\log^2(p)}{\sqrt{p}}$. With fixed $s$ and
  $\|x^*\|_1$, the weighted LASSO estimator has an error rate
  $\propto 1/\sqrt{p}$, while the error rate of the LASSO estimator
  $\propto\log(p)^2/\sqrt{p}$. To better show the relationship between
  the MSE and $1/{\sqrt{p}}$, we plot two additional lines
  $\propto 1/{\sqrt{p}}$ for the one-step estimators $\xW$ and $\xL$
  in Figure~\ref{fig:MSEvsp}. In Figure~\ref{fig:MSEvsp}, the MSE
  curve of $\xW$ follows the $6.5\times 10^5/\sqrt{p}$ curve almost
  perfectly, while the MSE curve of $\xL$ decreases (with $p$) more slowly
  than the $1.3\times 10^6/\sqrt{p}$ curve. This shows that the MSE of
  $\xW$ has a rate $\propto1/\sqrt{p}$, while the MSE of $\xL$ is
  slower than $1/\sqrt{p}$, as predicted by the
  theory. Figure~\ref{fig:MSEvsm} shows MSE vs $m$ for $p = 5000$,
  where results are averaged
  over $50$ trials. This plot demonstrates that for large $m$, the
  weighted LASSO and classical LASSO are nearly equivalent, while for
  the range of $m$ in \eqref{relationshipspourxi}, the Weighted
  LASSO has lower errors.  }
\begin{figure}[!ht]
\centering
\subfloat[MSE vs.\ $p$]
{\includegraphics[width=0.48\textwidth]{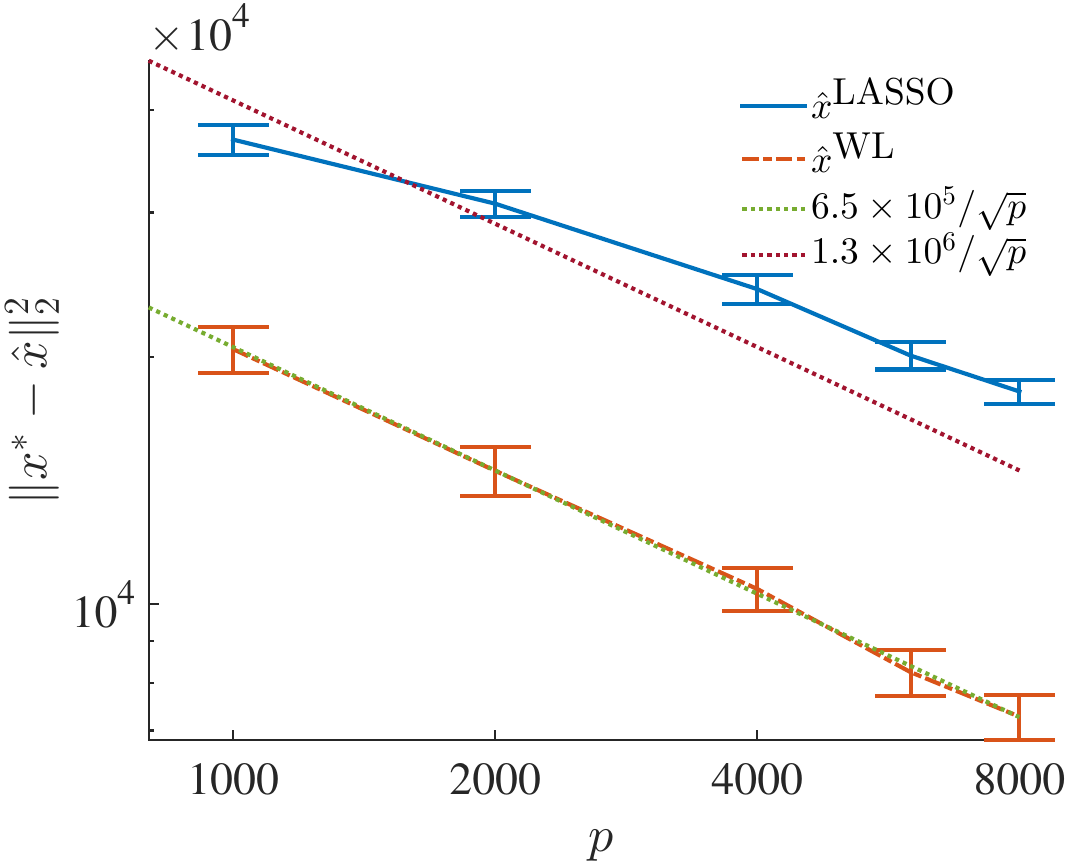} \label{fig:MSEvsp}}\hfill
\subfloat[MSE vs.\ $m$] 
{\includegraphics[width=.48\textwidth]{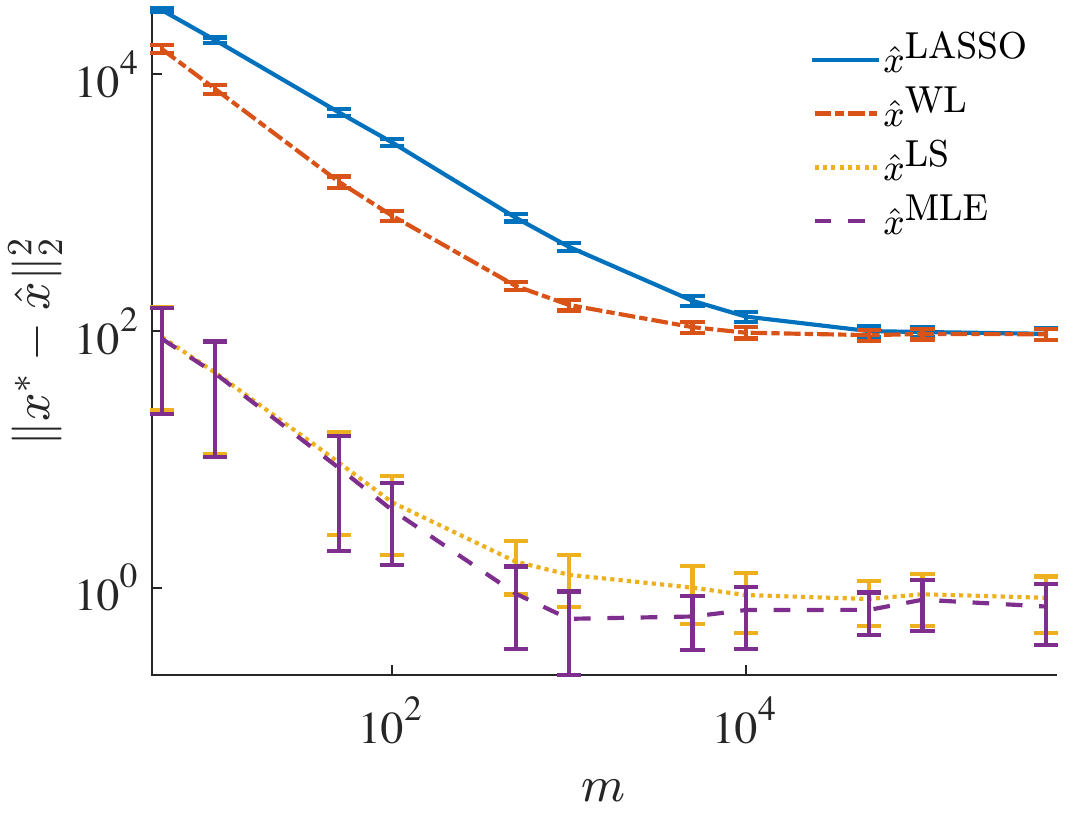} \label{fig:MSEvsm}}
\caption{Simulation results. \ref{fig:MSEvsp} shows MSE vs.\,$p$ with
  $m \propto \sqrt{p}\log{p}$ for the weighted LASSO and standard
  LASSO estimators. Weighted LASSO outperforms standard LASSO. The
  oracle estimators $\xls$ and $\xmle$ perform
  similarly. The MSE curve of $\xW$ follows
  the $6.5\times 10^5/\sqrt{p}$ curve almost perfectly, while the MSE
  curve of $\xL$ decreases (with $p$) more slowly than the
  $1.3\times 10^6/\sqrt{p}$ curve, showing that $\xW$ has a rate
  $\propto 1/\sqrt{p}$, while $\xL$ has a rate slower than
  $1/\sqrt{p}$. \ref{fig:MSEvsm} shows MSE vs $m$ for $p = 5000$ in the
  Poisson random convolution setting for genomics. Results averaged
  over $50$ trials. This plot demonstrates that for large $m$, the
  weighted LASSO and classical LASSO are nearly equivalent, while for
  the range of $m$ in   \eqref{relationshipspourxi}, the Weighted
  LASSO has lower errors. }
\label{fig:MSESim}
\end{figure}


\section{Discussion and Conclusions}

The data-dependent weighted LASSO method presented in this paper is a
novel approach to sparse inference in the presence of heteroscedastic
noise.  We show that a using concentration inequalities to learn
data-dependent weights leads to estimation errors which closely
approximate errors achievable by an oracle with knowledge of the true
signal support. 
To use this technique, concentration inequalities which account for
the noise distribution are used to set data-dependent weights which
satisfy the necessary assumptions with high probability.

In contrast to earlier work on sparse Poisson inverse problems
\cite{PCS_garvesh}, the estimator proposed here is computationally
tractable. In addition, earlier analyses required ensuring that the
product $Ax^*$ was bounded away from zero, which limited the
applicability of the analysis. Specifically, the random convolution
problem described in Section~\ref{sec:conv} could not be directly
analyzed using the techniques described in \cite{PCS_garvesh}.

Our technique can also yield immediate insight into the role of
background contamination. Consider a setting in which we observe
$$Y \sim \mathcal{P}(A x^* + b)$$
where $b \in \reals_+^n$ is a known (typically constant) background
intensity. In imaging, for instance, this would correspond to ambient
light or dark current effects. While $b$ contributes to the
noise variance, it does not provide any information about the unknown
signal $x^*$. Since $b$ is known, it can easily be subtracted from the
observations in the formation of $\tY$ and we can use exactly the
estimation framework described above ({\em e.g.}, the estimator
in~\eqref{eq:lasso}). However, because $b$ impacts the variance of the
observations, it will increase the value of $v$ in
Lemma~\ref{concPoisGen}, leading to a proportional increase in the
weights and hence the $\ell_2$ error decay rates. From here we can see
that the error decay rates will increase linearly with the amount of
background contamination.

It is worth noting that the core results of Section~\ref{sec:bounds}
do not use any probabilistic arguments and therefore do not rely at all
on Poisson noise assumptions. The
Poisson noise model is only used to derived data-dependent weights
that satisfy the necessary
assumptions high probability under the assumed
observation model. To extend our framework to new observation or noise
models, we would simply need to complete the following
(interdependent) tasks:
\begin{enumerate}
\item Determine a mapping from $A$ to $\tA$ which ensures $\tA$ satisfies
  Assumption~\ref{as:RE}.
\item Determine a mapping from $Y$ to $\tY$ so that
  $\expect[\tA^\T(\tY-\tA x^*)] = 0.$
\item Use concentration inequalities based on the assumed noise model
  to derive data-dependent weights which satisfy Assumption~\ref{as:dk}.
\end{enumerate}
Once these tasks are complete, the results of Section~\ref{sec:bounds}
can be immediately applied to compute recovery error rates.  Thus the
proposed weighted LASSO framework has potential for a variety of
settings and noise distributions.

\section{Acknowledgments}
We gratefully acknowledge the support of AFOSR award 14-AFOSR-1103,
NSF award CCF-1418976, the University of Nice Sophia Antipolis
Visiting Professor program and ANR 2011 BS01 010 01 projet 
Calibration.  The research of Vincent Rivoirard benefited from the
support of the {\it Chaire Economie et Gestion des Nouvelles 
  Donn\'ees}, under the auspices of Institut Louis Bachelier,
Havas-Media and Paris-Dauphine.

\bibliographystyle{unsrt}
\bibliography{Circle_Lasso_nohref}

\begin{thebibliography}{10}

\bibitem{willett:tmi03}
R.~M. Willett and R.~D. Nowak.
\newblock Platelets: a multiscale approach for recovering edges and surfaces in
  photon-limited medical imaging.
\newblock {\em Medical Imaging, IEEE Transactions on}, 22(3):332--350, 2003.

\bibitem{fesslerPCS}
D.~J. Lingenfelter, J.~A. Fessler, and Z.~He.
\newblock Sparsity regularization for image reconstruction with poisson data.
\newblock In {\em IS\&T/SPIE Electronic Imaging}, pages 72460F--72460F.
  International Society for Optics and Photonics, 2009.

\bibitem{schmidt2009optimal}
Taly~Gilat Schmidt.
\newblock Optimal “image-based” weighting for energy-resolved ct.
\newblock {\em Medical physics}, 36(7):3018--3027, 2009.

\bibitem{Kazik14}
K.~J. Borkowski, S.~P. Reynolds, D.~A. Green, U.~Hwang, R.~Petre,
  K.~Krishnamurthy, and R.~Willett.
\newblock Nonuniform expansion of the youngest galactic supernova remnant
  g1.9+0.3.
\newblock {\em The Astrophysical Journal Letters}, 790(2), 2014.
\newblock arXiv:1406.2287.

\bibitem{Kazik13}
K.~J. Borkowski, S.~P. Reynolds, D.~A. Green, U.~Hwang, R.~Petre,
  K.~Krishnamurthy, and R.~Willett.
\newblock Supernova ejecta in the youngest galactic supernova remnant g1.9+0.3.
\newblock {\em The Astrophysical Journal Letters}, 771(1), 2013.
\newblock arXiv:1305.7399.

\bibitem{starckPCS}
J.-L. Starck and J.~Bobin.
\newblock Astronomical data analysis and sparsity: from wavelets to compressed
  sensing.
\newblock {\em Proceedings of the IEEE}, 98(6):1021--1030, 2010.

\bibitem{laure}
L.~Sansonnet.
\newblock Wavelet thresholding estimation in a poissonian interactions model
  with application to genomic data.
\newblock {\em Scandinavian Journal of Statistics}, 41(1):200--226, 2014.

\bibitem{ElephantsMice}
C.~Estan and G.~Varghese.
\newblock New directions in traffic measurement and accounting: {F}ocusing on
  the elephants, ignoring the mice.
\newblock {\em ACM Transactions on Computer Systems}, 21(3):270--313, 2003.

\bibitem{CounterBraids}
Y.~Lu, A.~Montanari, B.~Prabhakar, S.~Dharmapurikar, and A.~Kabbani.
\newblock Counter {B}raids: {A} novel counter architecture for per-flow
  measurement.
\newblock In {\em Proc. ACM SIGMETRICS}, 2008.

\bibitem{PoissonCrime}
D.~W. Osgood.
\newblock {Poisson}-based regression analysis of aggregate crime rates.
\newblock {\em Journal of Quantitative Criminology}, 16(1), 2000.

\bibitem{socioscope}
J.-M. Xu, A.~Bhargava, R.~Nowak, and X.~Zhu.
\newblock Socioscope: Spatio-temporal signal recovery from social media.
\newblock {\em Machine Learning and Knowledge Discovery in Databases}, 7524,
  2012.

\bibitem{PCS_garvesh}
X.~Jiang, G.~Raskutti, and R.~Willett.
\newblock Minimax optimal rates for poisson inverse problems with physical
  constraints.
\newblock To appear in {\em IEEE Transactions on Information Theory},
  arXiv:1403:6532, 2014.

\bibitem{pcs}
M.~Raginsky, R.~M. Willett, Z.~T. Harmany, and R.~F. Marcia.
\newblock Compressed sensing performance bounds under poisson noise.
\newblock {\em Signal Processing, IEEE Transactions on}, 58(8):3990--4002,
  2010.

\bibitem{expFamCS}
I.~Rish and G.~Grabarnik.
\newblock Sparse signal recovery with exponential-family noise.
\newblock In {\em Allerton Conference on Communication, Control, and
  Computing}, 2009.

\bibitem{expander_PCS}
M.~Raginsky, S.~Jafarpour, Z.~Harmany, R.~Marcia, R.~Willett, and
  R.~Calderbank.
\newblock Performance bounds for expander-based compressed sensing in {P}oisson
  noise.
\newblock {\em IEEE Transactions on Signal Processing}, 59(9), 2011.
\newblock arXiv:1007.2377.

\bibitem{rohban2015minimax}
M.~H. Rohban, D.~Motamedvaziri, and V.~Saligrama.
\newblock Minimax optimal sparse signal recovery with poisson statistics.
\newblock {\em arXiv preprint arXiv:1501.05200}, 2015.

\bibitem{IPR}
S.~Ivanoff, F.~Picard, and V.~Rivoirard.
\newblock Adaptive lasso and group-lasso for functional poisson regression.
\newblock {\em Journal of Machine Learning Research}, 17(55):1--46, 2016.

\bibitem{LASSO}
R.~Tibshirani.
\newblock Regression shrinkage and selection via the lasso.
\newblock {\em J. Roy. Statist. Soc. Ser. B}, 58(1):267--288, 1996.

\bibitem{CS:candes2}
E.~J. Candes and T.~Tao.
\newblock Near-optimal signal recovery from random projections: Universal
  encoding strategies?
\newblock {\em Information Theory, IEEE Transactions on}, 52(12):5406--5425,
  2006.

\bibitem{CS:donoho}
D.~L. Donoho.
\newblock Compressed sensing.
\newblock {\em IEEE Transactions on Information Theory}, 52(4):1289--1306,
  2006.

\bibitem{BRT}
P.~J. Bickel, Y.~Ritov, and A.~B. Tsybakov.
\newblock Simultaneous analysis of {Lasso} and dantzig selector.
\newblock {\em The Annals of Statistics}, pages 1705--1732, 2009.

\bibitem{bunea}
F.~Bunea, A.~B. Tsybakov, and M.~H. Wegkamp.
\newblock Aggregation for {G}aussian regression.
\newblock {\em Ann. Statist.}, 35(4):1674--1697, 2007.

\bibitem{Lounici}
K.~Lounici.
\newblock Sup-norm convergence rate and sign concentration property of {L}asso
  and {D}antzig estimators.
\newblock {\em Electron. J. Stat.}, 2:90--102, 2008.

\bibitem{van2008high}
S.~A. Van~de Geer.
\newblock High-dimensional generalized linear models and the {LASSO}.
\newblock {\em The Annals of Statistics}, 36(2):614--645, 2008.

\bibitem{Buhlmann:2011}
P.~B{\"u}hlmann and S.~van~de Geer.
\newblock {\em Statistics for high-dimensional data}.
\newblock Springer Series in Statistics. Springer, Heidelberg, 2011.
\newblock Methods, theory and applications.

\bibitem{PBHT}
Debashis Paul, Eric Bair, Trevor Hastie, and Robert Tibshirani.
\newblock ``{P}reconditioning'' for feature selection and regression in
  high-dimensional problems.
\newblock {\em Ann. Statist.}, 36(4):1595--1618, 2008.

\bibitem{Wauthier_acomparative}
Fabian~L. Wauthier, Nebojsa Jojic, and Michael~I. Jordan.
\newblock A comparative framework for preconditioned lasso algorithms.
\newblock In {\em Advances in Neural Information Processing Systems}, pages
  1061--1069, 2013.

\bibitem{HJ}
J~C. Huang and N.~Jojic.
\newblock Variable selection through correlation sifting.
\newblock In {\em Proceedings of the 15th Annual international conference on
  Research in computational molecular biology}, pages 103--126. Springer, 2011.

\bibitem{vGBZ11}
S.~van~de Geer, P.~B{\"u}hlmann, and S.~Zhou.
\newblock The adaptive and the thresholded {Lasso} for potentially misspecified
  models (and a lower bound for the {Lasso}).
\newblock {\em Electronic Journal of Statistics}, 5:688--749, 2011.

\bibitem{Juditsky}
A.~Juditsky and A.~Nemirovski.
\newblock Accuracy guarantees for {$\ell_1$}-recovery.
\newblock {\em IEEE Trans. Inform. Theory}, 57(12):7818--7839, 2011.

\bibitem{RIPCandes}
E.~Cand{\`e}s.
\newblock {The restricted isometry property and its implications for compressed
  sensing}.
\newblock {\em Comptes Rendus Mathematique}, 346(9-10):589--592, May 2008.

\bibitem{negahban2009unified}
S.~Negahban, B.~Yu, M.~J. Wainwright, and P.~K. Ravikumar.
\newblock A unified framework for high-dimensional analysis of $ m $-estimators
  with decomposable regularizers.
\newblock In {\em Advances in Neural Information Processing Systems}, pages
  1348--1356, 2009.

\bibitem{candes2007dantzig}
E.~Candes and T.~Tao.
\newblock The dantzig selector: Statistical estimation when p is much larger
  than n.
\newblock {\em The Annals of Statistics}, pages 2313--2351, 2007.

\bibitem{BLPR11}
K.~Bertin, E.~Le~Pennec, and V.~Rivoirard.
\newblock Adaptive dantzig density estimation.
\newblock {\em Annales de l'institut Henri Poincar{\'e} (B)}, 47:43--74, 2011.

\bibitem{HRBR12}
N.~R. Hansen, P.~Reynaud-Bouret, and V.~Rivoirard.
\newblock Lasso and probabilistic inequalities for multivariate point
  processes.
\newblock {\em Bernoulli}, 21(1):83--143, 2015.

\bibitem{HMZ08}
J.~Huang, S.~Ma, and C.-H. Zhang.
\newblock Adaptive {Lasso} for sparse high-dimensional regression models.
\newblock {\em Statistica Sinica}, 18(4):1603, 2008.

\bibitem{Zou06}
H.~Zou.
\newblock The adaptive {Lasso} and its oracle properties.
\newblock {\em Journal of the American statistical association},
  101(476):1418--1429, 2006.

\bibitem{DJ94}
D.~L. Donoho and I.~M. Johnstone.
\newblock Ideal spatial adaptation by wavelet shrinkage.
\newblock {\em Biometrika}, 81(3):425--455, 1994.

\bibitem{JLL04}
A.~Juditsky and S.~Lambert-Lacroix.
\newblock On minimax density estimation on $\mathbb{R}$.
\newblock {\em Bernoulli}, 10(2):187--220, 2004.

\bibitem{RBR10}
P.~Reynaud-Bouret and V.~Rivoirard.
\newblock Near optimal thresholding estimation of a poisson intensity on the
  real line.
\newblock {\em Electronic journal of statistics}, 4:172--238, 2010.

\bibitem{RBRTM11}
P.~Reynaud-Bouret, V.~Rivoirard, and C.~Tuleau-Malot.
\newblock Adaptive density estimation: a curse of support?
\newblock {\em Journal of Statistical Planning and Inference}, 141(1):115--139,
  2011.

\bibitem{riceCamera}
M.~F. Duarte, M.~A. Davenport, D.~Takhar, J.~N. Laska, T.~Sun, K.~F. Kelly, and
  R.~G. Baraniuk.
\newblock Single pixel imaging via compressive sampling.
\newblock {\em IEEE Sig. Proc. Mag.}, 25(2):83--91, 2008.

\bibitem{hastie2015statistical}
T.~Hastie, R.~Tibshirani, and M.~Wainwright.
\newblock {\em Statistical learning with sparsity: the lasso and
  generalizations}.
\newblock CRC press, 2015.

\bibitem{raskutti2010restricted}
G.~Raskutti, M.~J. Wainwright, and B.~Yu.
\newblock Restricted eigenvalue properties for correlated gaussian designs.
\newblock {\em Journal of Machine Learning Research}, 11(Aug):2241--2259, 2010.

\bibitem{li2016minimax}
Y.~Li and G.~Raskutti.
\newblock Minimax optimal convex methods for poisson inverse problems under
  $\ell_q$-ball sparsity.
\newblock {\em arXiv preprint arXiv:1604.08943}, 2016.

\bibitem{STM}
L.~Sansonnet and C.~Tuleau-Malot.
\newblock A model of {P}oissonian interactions and detection of dependence.
\newblock {\em Stat. Comput.}, 25(2):449--470, 2015.

\bibitem{GS}
G.~Gusto and S.~Schbath.
\newblock F{ADO}: a statistical method to detect favored or avoided distances
  between occurrences of motifs using the {H}awkes' model.
\newblock {\em Stat. Appl. Genet. Mol. Biol.}, 4:Art. 24, 28 pp. (electronic),
  2005.

\bibitem{CSWH}
L.~Carstensen, A.~Sandelin, O.~Winther, and N.~R. Hansen.
\newblock Multivariate hawkes process models of the occurrence of regulatory
  elements and an analysis of the pilot encode regions.
\newblock {\em BMC Bioinformatics}, 11(456), 2010.

\bibitem{RombergToeplitz}
J.~Romberg.
\newblock Compressive sensing by random convolution.
\newblock {\em SIAM Journal on Imaging Sciences}, 2(4):1098--1128, 2009.

\bibitem{CCA}
R.~F. Marcia and R.~M. Willett.
\newblock Compressive coded aperture superresolution image reconstruction.
\newblock In {\em Acoustics, Speech and Signal Processing, 2008. ICASSP 2008.
  IEEE International Conference on}, pages 833--836. IEEE, 2008.

\bibitem{CAKE}
Z.~T. Harmany, R.~F. Marcia, and R.~M. Willett.
\newblock Spatio-temporal compressed sensing with coded apertures and keyed
  exposures.
\newblock {\em arXiv:1111.7247}, 2011.

\bibitem{CandesPlan}
E.~J. Candes and Y.~Plan.
\newblock A probabilistic and ripless theory of compressed sensing.
\newblock {\em Information Theory, IEEE Transactions on}, 57(11):7235--7254,
  2011.

\bibitem{stflour}
P.~Massart.
\newblock {\em Concentration inequalities and model selection}, volume~6.
\newblock Springer, 2007.

\bibitem{vershynin2010introduction}
R.~Vershynin.
\newblock Introduction to the non-asymptotic analysis of random matrices.
\newblock {\em arXiv preprint arXiv:1011.3027}, 2010.

\bibitem{ustats}
C.~Houdr{\'e} and P.~Reynaud-Bouret.
\newblock Exponential inequalities, with constants, for u-statistics of order
  two.
\newblock In {\em Stochastic inequalities and applications}, pages 55--69.
  Springer, 2003.

\bibitem{GineNickl}
E.~Gin\'e and R.~Nickl.
\newblock {\em Mathematical foundations of inifnite dimensional statistical
  models}.
\newblock Cambridge University Press, 2016.

\bibitem{mendelson2008uniform}
S.~Mendelson, A.~Pajor, and N.~Tomczak-Jaegermann.
\newblock Uniform uncertainty principle for bernoulli and subgaussian
  ensembles.
\newblock {\em Constructive Approximation}, 28(3):277--289, 2008.

\bibitem{MR1857312}
E.~Gin{\'e}, R.~Lata{\l}a, and J.~Zinn.
\newblock Exponential and moment inequalities for {$U$}-statistics.
\newblock In {\em High dimensional probability, {II} ({S}eattle, {WA}, 1999)},
  volume~47 of {\em Progr. Probab.}, pages 13--38. Birkh\"auser Boston, Boston,
  MA, 2000.

\end{thebibliography}

\appendix
\numberwithin{equation}{section}
\numberwithin{lemma}{section}
\numberwithin{theorem}{section}
\numberwithin{remark}{section}

\section{Proofs of the LASSO bounds of Section~\ref{sec:bounds}}

\label{pf:l2}
\revision{The weighted LASSO regularizer, $\|Dx\|_1$ is {\em decomposable} in the
sense of \cite{negahban2009unified}, and hence we can leverage
variants of the results developed in that paper to characterize the
performance of the weighted LASSO estimator. The results of
\cite{negahban2009unified} applied na\"ively to the proposed estimator
results in rates equivalent to those associated with the classical
(unweighed) LASSO. Thus several technical details of the analysis must
be adjusted to derive tight bounds for the weighted LASSO
estimator. This section contains those details for the sake of
completeness.}

Before proving Proposition~\ref{th:L2}, we need the following supporting
lemma.
\begin{lemma} \label{lem:gamma2}
Consider any optimal solution $\xW$ to the weighted LASSO
optimization problem with $\gamma > 2$.
Then, under  Assumption~\asref{as:dk}{\{d_k\}_k}, for any $S  \subseteq \{1,\ldots,p\}$,
the error $\Delta = \xW - x^*$ satisfies
$$\|D\Delta_{S^c}\|_1 \le \frac{\gamma +2}{\gamma-2}\|D
\Delta_S\|_1 + \frac{2\gamma}{\gamma-2} \|D x^*_{S^c}\|_1.$$
\end{lemma}
\begin{remark} Lemma~\ref{lem:gamma2} implies
\begin{equation}
\|D\Delta\|_1= \|D\Delta_{S}\|_1+\|D\Delta_{S^c}\|_1 \le \frac{2\gamma}{\gamma-2}\|D
\Delta_S\|_1 + \frac{2\gamma}{\gamma-2} \|D x^*_{S^c}\|_1.
\label{eq:lem1b}
\end{equation}
\end{remark} 

\begin{proof}{Lemma~\ref{lem:gamma2}}
Let $\Delta := \xW-x^*$. Our proof follows the structure of
\cite[Lemma 3]{negahban2009unified}:
\begin{align*}
\|\tY - \tA \xW\|_2^2 - \|\tY - \tA x^*\|_2^2 =& \|\tY - \tA x^*
-\tA\Delta\|_2^2 - \|\tY - \tA x^*\|_2^2\\ =& \|\tY - \tA x^*\|_2^2 -
2 \ave{\tY - \tA x^*, \tA \Delta} + \|\tA \Delta\|_2^2 - \|\tY - \tA
x^*\|_2^2\\ =&- 2 \ave{\tY - \tA x^*, \tA \Delta} + \|\tA \Delta\|_2^2
\\
\geq&- 2| \ave{D^{-1}\tA^\T(\tY - \tA x^*), D\Delta}|\\
\geq&-2\|D\Delta\|_1,
\end{align*}
by Assumption~\asref{as:dk}{\{d_k\}_k}. 
Now by the ``basic inequality'' we have
$$\|\tY - \tA \xW\|_2^2 - \|\tY - \tA x^*\|_2^2 \le
\gamma\left[\|Dx^*\|_1 - \|D(x^*+\Delta)\|_1\right]$$
Therefore,
$$-2\|D\Delta\|_1\leq \gamma\left[\|Dx^*\|_1 - \|D(x^*+\Delta)\|_1\right]$$
and since
\begin{eqnarray*}
\|D(x^*+\Delta)\|_1&\geq& \|D(x^*_S+\Delta_{S^c})\|_1-\|D(x^*_{S^c}+\Delta_S)\|_1\\
&\geq&\|Dx^*_S\|_1+\|D\Delta_{S^c}\|_1-\|Dx^*_{S^c}\|_1-\|D\Delta_S\|_1
\end{eqnarray*}
we obtain
$$-2\|D\Delta_S\|_1-2\|D\Delta_{S^c}\|_1\leq \gamma\|Dx^*\|_1 - \gamma\left(\|Dx^*_S\|_1+\|D\Delta_{S^c}\|_1-\|Dx^*_{S^c}\|_1-\|D\Delta_S\|_1\right)$$
and
$$\|D\Delta_{S^c}\|_1 \le \frac{\gamma +2}{\gamma-2}\|D
\Delta_S\|_1 + \frac{2\gamma}{\gamma-2} \|D x^*_{S^c}\|_1.$$
\end{proof}

\begin{proof}{Proposition~\ref{th:L2}}
We still denote $\Delta := \xW-x^*$. By the ``basic inequality'' we have
$$\|\tY - \tA \xW\|_2^2 - \|\tY - \tA x^*\|_2^2 \le
\gamma\left[\|Dx^*\|_1 - \|D(x^*+\Delta)\|_1\right].$$
This gives 
$$- 2 \ave{\tY - \tA x^*, \tA \Delta} + \|\tA \Delta\|_2^2 \le
\gamma\left[\|Dx^*\|_1  - \|D(x^*+\Delta)\|_1\right]$$
or
\begin{align*}
\|\tA \Delta\|_2^2 \le&  2 \ave{\tY - \tA x^*, \tA \Delta}
                        +\gamma\left[\|Dx^*\|_1  - \|D(x^*+\Delta)\|_1\right]\\
\le& 2\|D^{-1}(\tA^\T(\tY - \tA x^*))\|_\infty \|D\Delta\|_1  +\gamma\left[\|Dx^*\|_1  - \|D(x^*+\Delta)\|_1\right]\\
\le&2\|D\Delta\|_1 +\gamma\left[\|Dx^*\|_1  -
     \|D(x^*+\Delta)\|_1\right]\\
=&2\|D\Delta\|_1 +\gamma \|D x^*_S\|_1 + \gamma
     \|Dx_{S^c}^*\|_1 - \gamma\|D(x^*+\Delta)_S\|_1-\gamma\|D(x^*+\Delta)_{S^c}\|_1\\
\le&2\|D\Delta\|_1 +\gamma \|D \Delta_S \|_1 + \gamma
     \|Dx_{S^c}^*\|_1\\
\le&\rho_{\gamma}\left(\|D\Delta_S\|_1 + \|Dx^*_{S^c}\|_1\right)
\end{align*}
where for the last step we use \eqref{eq:lem1b} and the definition of $\rho_{\gamma}$. We now provide a lower bound to the left hand side. First we use \eqref{eq:lem1b} to bound
\begin{align*}
\|\Delta\|_1 \le& \frac{\tilde\rho_{\gamma}}{\dmin}\left(\|D\Delta_{S}\|_1 +
                        \|Dx^*_{S^c}\|_1\right)
\le\frac{\tilde\rho_{\gamma}}{\dmin}\left(\|d_S\|_2 \|\Delta\|_2 + \|Dx^*_{S^c}\|_1\right)
\end{align*}
where
$$\tilde\rho_{\gamma}:=\frac{2\gamma }{(\gamma-2)}\quad\mbox{and}\quad \|d_S\|_2 := \left( \sum_{k \in S} d_k^2 \right)^{1/2}.$$
Combining this with Assumption~\asref{as:RE}{\kb,\ka} gives
\begin{align*}
\|\tA \Delta\|_2 \ge& \ka \|\Delta\|_2 - \frac{\tilde\rho_{\gamma}
  \kb}{\dmin}[\|d_S\|_2 \|\Delta\|_2 + \|Dx^*_{S^c}\|_1]\\
=& \|\Delta\|_2 \left(\ka - \frac{\tilde\rho_{\gamma}\kb}{\dmin}\|d_S\|_2\right) 
- \frac{\tilde\rho_{\gamma}\kb}{\dmin} \|Dx^*_{S^c}\|_1.
\end{align*}
Recall by assumption that
$$\ka - \tilde\rho_{\gamma}\kb \frac{\|d_S\|_2}{\dmin} \geq \eps;$$
then 
$$
\eps \|\Delta\|_2 \le \|\tA\Delta\|_2 +
                       \frac{\tilde\rho_{\gamma}\kb}{\dmin}\|Dx^*_{S^c}\|_1$$
and using the previous control of $\|\tA\Delta\|_2^2$, we obtain
\begin{align*}
\|\Delta\|_2^2 \le&\frac{2}{\eps^2} \|\tA\Delta\|_2^2+\frac{2\tilde\rho_{\gamma}^2\kb^2}{\eps^2\dmin^2}\|Dx^*_{S^c}\|_1^2\\
\leq&
  \frac{2\rho_{\gamma}}{\eps^2}\|d_S\|_2\|\Delta\|_2+\frac{2\rho_{\gamma}}{\eps^2}\|Dx^*_{S^c}\|_1
  +  \frac{2\tilde\rho_{\gamma}^2\kb^2}{\eps^2\dmin^2}\|Dx^*_{S^c}\|_1^2\\
\le&  \frac{2\rho_{\gamma}}{\eps^2}\|d_S\|_2\|\Delta\|_2+ \sigma_S^2
\end{align*}
where 
$$\sigma_S^2 := \frac{2\rho_{\gamma}}{\eps^2}\|Dx^*_{S^c}\|_1
  +  \frac{2\tilde\rho_{\gamma}^2\kb^2}{\eps^2\dmin^2}\|Dx^*_{S^c}\|_1^2.$$
Recall $y^2 - by - c \le 0$ implies $y^2 \le b^2 + 2c$ for any
$y$. Thus,
\begin{align*}
\| \Delta\|_2^2 \leq& \frac{4\rho_{\gamma}^2}{\eps^4}\|d_S\|_2^2 + 2\sigma_S^2 \\
\leq&\frac{4\rho_{\gamma}^2}{\eps^4}\|d_S\|_2^2 +\frac{4\rho_{\gamma}}{\eps^2}\|Dx^*_{S^c}\|_1 +    \frac{4\tilde\rho_{\gamma}^2\kb^2}{\eps^2\dmin^2}\|Dx^*_{S^c}\|_1^2.
\end{align*}
We conclude by observing that $\tilde\rho_{\gamma}\leq\rho_{\gamma}.$
\end{proof}

\section{Concentration inequality for data-dependent weights (proof of Lemma \ref{concPoisGen})}
\label{ap:weights}
The proof of \eqref{unC} is classical and follows the lines of the
proof of Bernstein's inequality. Let $Z=R^\T Y$ and $z=R^\T A x^*$.
Conditioned on the sensing matrix $A$, the $Y_\ell$'s are independent
Poisson variables of mean $\sum_{k=1}^{p} a_{\ell,k} x^*_k.$ Therefore
for all $\lambda>0$ (eventually depending only on the sensing matrix
$A$)
\begin{align*}
\E\left(e^{\lambda (Z-z)}\big| A \right) =&
\prod_{\ell=1}^{p} \E\left(e^{\lambda R_\ell
                                             \left[Y_\ell-\sum_{k=1}^{p}
                                             a_{\ell,k} x^*_k \right]}\big|A\right) = \prod_{\ell=1}^{p}  \exp\left[ \left(e^{\lambda R_\ell}-\lambda
    R_\ell-1\right) \sum_{k=1}^{p} a_{\ell,k} x^*_k \right].\\
\end{align*}
If $ \lambda <(3/b)$, then by classical computations (see \cite{stflour} for instance), for all $\ell$,
$$\big|e^{\lambda r_\ell}-\lambda r_\ell-1\big| \leq \frac{\lambda^2 r_\ell^2 /2}{1-\lambda b/3}.$$
Therefore, if $ \lambda <(3/b)$,
$$\E\left(e^{\lambda (Z-z)}\big|A\right) \leq \exp\left(\frac{\lambda^2 v /2}{1-\lambda b/3}\right).$$
Hence by Markov's inequality, for all $u>0$
$$\P(Z-z\geq u)\leq  \exp\left(\frac{\lambda^2 v /2}{1-\lambda b/3}-\lambda u\right).$$
It remains to optimize in $\lambda$ and conclude as in Bernstein's
inequality (see \cite{stflour}). \revision{More precisely, the upper bound
  in $\lambda$ is minimal if
$$ \lambda = \frac{2 \frac{bu}{3}+ v - \sqrt{v^2+2 \frac{buv}{3} }}{2 \left(\frac{b^2u}{9}+\frac{vb}{6}\right)},$$
which gives
$$\P(Z-z\geq u) \leq \exp\left(\frac{\sqrt{v^2+2 \frac{buv}{3} }- v - \frac{bu}{3}}{b^2/9}\right).$$
We want the upper bound to be equal to $e^{-\theta}$, which gives by inversion
$$u = \sqrt{2 v\theta}+ \frac{b \theta}{3}.$$}

For \eqref{deuxC} it is sufficient to apply \eqref{unC} to both $R$ and $-R$. For \eqref{troisC} it is sufficient to apply \eqref{unC} to $-R_2$ and for \eqref{quatreC}, to combine both \eqref{deuxC} and \eqref{troisC}.

\section{Validation of assumptions for Bernoulli sensing of Section~\ref{sec:bernoulli}}
\subsection{Rescaling and recentering}
\label{scalingBern}
First let us prove that $\E( \tG)=I_p$,  with $\tG={\tA}^\T{\tA}.$
Indeed, the $(k,k')$ element of $\tG$ is
$$
\tG_{k,k'}= \frac{\sum_{\ell=1}^n (a_{\ell,k}-q)(a_{\ell,k'}-q)}{n q (1-q)}.
$$
Hence
$\E(\tG_{k,k'}) = 0$ if $k\not = k'$ and $\E(\tG_{k,k})=1$.
Next let $Z= {\tA}^\T({\tY}-{\tA}x^*)$ and let us prove that
$\E(Z)=0$.
\begin{align*}
Z=& \frac{1}{n q(1-q)}(A^\T-q  \ones_{p \times 1} \ones_{\nObs \times
  1}^\T) \left(\frac{nY-(\sum_{k=1}^p Y_k) \ones_{\nObs \times
  1}}{n-1}-(A-q  \ones_{\nObs \times 1} \ones_{p \times
  1}^\T)x^*\right)\\
  =&\frac{1}{n q(1-q)}\left( \frac{n}{n-1} A^\T Y- \frac{\sum_{k=1}^p Y_k}{n-1} A^\T \ones_{\nObs \times 1} - A^\T A x^* + \right.\\
& \qquad\qquad\qquad \left. + q  \norm{x^*}_1 A^\T \ones_{\nObs \times
  1} + q \ones_{p \times 1} \ones_{\nObs \times
  1}^\T A x^* - q^2 \norm{x^*}_1 \ones_{p \times 1} \ones_{\nObs \times
  1}^\T \ones_{\nObs \times
  1}\right)\\
  =& T_1+ T_2
\end{align*}
with 
$$T_1 = \frac{1}{n q(1-q)}\left(  \frac{n}{n-1} A^\T - \frac{A^\T \ones_{\nObs \times
  1}\ones_{\nObs \times
  1}^\T}{n-1}\right) (Y-Ax^*)$$
  and
  $$T_2= \frac{1}{n q(1-q)}\left(\frac{A^\T Ax^*-A^\T \ones_{\nObs \times
  1}\ones_{\nObs \times
  1}^\T A x^*}{n-1}+ q  \norm{x^*}_1 A^\T \ones_{\nObs \times
  1}+  q \ones_{p \times 1} \ones_{\nObs \times
  1}^\T A x^* - q^2 n \norm{x^*}_1 \ones_{p \times 1} \right).
$$
Since $\E(Y|A)=Ax^*$, $\E(T_1|A)=$ and therefore $\E(T_1)=0$.

Next the $k$th element of $T_2$ only depends on $A$ and satisfies
\begin{align*}
n q (1-q) T_{2,k}=& \frac{1}{n-1}\left(\sum_{\ell=1}^n\sum_{k'=1}^p a_{\ell,k} a_{\ell,k'} x_k' -\sum_{\ell,\ell'=1}^n \sum_{k'=1}^p a_{\ell,k}a_{\ell',k'} x_{k'}\right)+q\sum_{\ell=1}^n a_{\ell,k}\sum_{k'=1}^p x_{k'} + \\
& \qquad \qquad q\sum_{\ell'=1}^n\sum_{k'=1}^p a_{\ell',k'} x_{k'}- q^2 n \sum_{k'=1}^p x_{k'}\\
=& \frac{-1}{n-1} \sum_{\ell=1}^n\sum_{\ell' \not = \ell} \sum_{k'=1}^p (a_{\ell,k}-q)(a_{\ell',k'}-q) x_{k'}.
\end{align*}
Hence every element of $T_2$ is a degenerate U-statistics of order 2  and $\E(T_2)=0$. Note that $T_2$ can also be seen as
$T_2=\mathbb{M} x^*$ with
$$\mathbb{M}_{k,k'}=\frac{-1}{(n-1)n q (1-q)} \sum_{\ell=1}^n\sum_{\ell' \not = \ell} (a_{\ell,k}-q)(a_{\ell',k'}-q).$$

\subsection{Assumption~\ref{as:RE} holds (proof of Proposition~\ref{REnewCS})}
We use the following definition and
lemma:
\begin{Def}[Zhou 2009, Definition 1.3] Let $Y$ be a random
  vector in $\reals^p$; $Y$ is called isotropic if for every $y \in
  \reals^p$, $\expect|\ave{Y,y}|^2 = \|y\|_2^2$, and is $\psi_2(\alpha)$, for $\alpha$ a positive constant, if for every $y \in \reals^p$
$$\|\ave{Y,y}\|_{\psi_2} \le \alpha \|y\|_2$$
where for a random variable $X \in \reals$
$$\|X\|_{\psi_2} := \inf\{ t: \expect \exp(X^2/t^2) \le 2 \}.$$
\end{Def}
\begin{lemma}[Theorem 3 in Li and Raskutti, Minimax optimal convex
  methods for Poisson inverse problems under $\ell_q$-ball sparsity]
  There exist positive constants $c,c^{\prime},c^{\prime \prime}$ for which the following
  holds. Let $B$ be an isotropic $\psi_2(\alpha)$ random vector of $\R^p$. Let $B_1,\ldots,B_n \in \reals^p$ be
  independent, distributed according to $B$ and define $\Gamma =
  \sum_{i=1}^n \ave{ B_i, \cdot } e_i$, where $e_i$ is the $i$th element of the standard basis of $\R^n$. Then with probability
  at least $1-c^{\prime} \exp(-c^{\prime \prime} n)$, for all $x \in \reals^p$ we have
$$\frac{\|x\|_2}{4} - c \alpha^2 \sqrt{\frac{\log p}{n}} \|x\|_1 \le
\frac{\|\Gamma x\|_2}{\sqrt{n}}.$$
\end{lemma}
We apply this lemma with
$$\Gamma=\sqrt{n}\times\tA.$$
Therefore, for any $i$ and any $j$,
$$(B_i)_j=\frac{A_{ij}-q}{\sqrt{q(1-q)}}.$$
and we have for any $x\in\R^p$ and any $i$,
$$\E\left[\langle B_i,x\rangle^2\right]=\Var\left(\sum_{j=1}^p(B_i)_jx_j\right)=\|x\|_2^2,$$
so the $B_i$'s are isotropic.
Furthermore, using  \cite[Lemma 5.9]{vershynin2010introduction}
$$
\|\langle B_i,x\rangle\|_{\psi_2}^2\leq \square\sum_{j=1}^p \|x_j(B_i)_j\|_{\psi_2}^2\leq\square\sum_{j=1}^px_j^2 \|(B_i)_j\|_{\psi_2}^2.
$$
And since 
$$|(B_i)_j|\leq\frac{1}{\sqrt{q(1-q)}},$$
using Lemma 5.5 and Example 5.8 of \cite{vershynin2010introduction}, we obtain
$$\|(B_i)_j\|_{\psi_2}\leq\frac{1}{\sqrt{q(1-q)}}$$
and
$$
\|\langle B_i,x\rangle\|_{\psi_2}^2\leq\square\times\frac{\|x\|_2^2}{q(1-q)}.$$
So, in our setting, the $B_i$'s are $\psi_2(\alpha)$ with
 $$\alpha = \frac{1}{\sqrt{q(1-q)}}.$$ 
So, with probability
  at least $1-c^{\prime} \exp(-c^{\prime \prime} n)$, for all $x \in \reals^p$
$$\|\tA x\|_2 = \frac{\|\Gamma x\|_2}{\sqrt{n}} \ge \frac{\|x\|_2}{4}
- \frac{c}{q (1-q)} \sqrt{\frac{\log p}{n}} \|x\|_1,$$ and Assumption~$\ref{as:RE}({\kb,\ka})$ holds with
$\ka = \frac{1}{4}$ and $\kb = \frac{c}{q (1-q)} \sqrt{\frac{\log p}{n}}$. 

\subsection{Proofs for data-dependent weights}
\begin{proposition}\label{dBern}  
Let $$W=\max_{u,k=1,\ldots,p} w(u,k)$$ with
$$w(u,k)=\frac{1}{n^2(n-1)^2 q^2(1-q)^2}\sum_{\ell=1}^\nObs a_{\ell,u} \left(n a_{\ell,k} -\sum_{\ell'=1}^\nObs a_{\ell',k}\right)^2.$$
Then if $ nq \geq 12 \max(q,1-q) \log(p)$, \revision{and if $p\geq 2$} then there exists absolute constants  $c,c^{\prime}$ such that with probability larger than $1-c^{\prime}/p$, the choice
\begin{equation*}
d =  \sqrt{6W\log(p)}  \sqrt{\hat{N}} +
\frac{\log(p)}{(n-1)q(1-q)}+ c \left(\frac{3\log(p)}{n} +
  \frac{9\max(q^2,(1-q)^2)}{n^2q(1-q)}\log(p)^2\right)  \hat{N},
\end{equation*}
where $\hat{N}$ is an estimator of $\norm{x^*}_1$ given by
$$\hat{N}=\frac{1}{nq-\sqrt{6nq(1-q)\log(p)}-\max(q,1-q)\log(p)} \left(\sqrt{\frac{3\log(p) }{2}}+\sqrt{\frac{5\log(p)}{2}+ \sum_{\ell=1}^n Y_\ell }\right)^2.
$$
satisfies Assumption \asref{as:dk}{d}.
\revision{Moreover one can take $c=126$ as long as $n\geq 20$.}
\end{proposition}
 
\begin{proposition}\label{bounddBern}
  There exists some absolute constant $\kappa$ such that if
$$nq^2(1-q) \geq \kappa \log(p)$$
then there exists a positive constant $C$ such that with
probability larger than $1-C/p$
$$ d \simeq \sqrt{\frac{\log(p) \norm{x^*}_1}{n\min(q,1-q)}} +
  \frac{\log(p)\norm{x^*}_1}{n}+
  \frac{\log(p)}{nq(1-q)}.$$
\end{proposition}

\begin{proposition}\label{dkBern}  
With the same notations and assumptions as Proposition \ref{dBern},
there exists absolute constants  $c,c^{\prime}$ such that with probability larger than $1-c^{\prime}/p$, the choice (depending on $k$)
\begin{align} d_k =   \sqrt{6\log(p)}  \left(\sqrt{\frac{3\log(p)
  }{2(n-1)^2q^2(1-q)^2}}+\sqrt{\frac{5\log(p)}{2(n-1)^2q^2(1-q)^2}+
  V_k^\T Y }\right) +\frac{\log(p)}{(n-1)q(1-q)}+ \\+c
  \left(\frac{3\log(p)}{n} +
  \frac{9\max(q^2,(1-q)^2)}{n^2q(1-q)}\log(p)^2\right)  \hat{N},
\end{align}
with the vector $V_k$ of size $n$ given by $$V_{k,\ell}=\left(\frac{na_{\ell,k}-\sum_{\ell'=1}^\nObs a_{\ell',k}}{n(n-1)q(1-q)}\right)^2$$
satisfies Assumption \asref{as:dk}{\{d_k\}_k}.
\revision{Moreover one can take $c=126$ as long as $n\geq 20$.}
\end{proposition}

\begin{proposition}\label{bounddkBern}
  There exists some absolute constant $\kappa$ such that if
$$nq^2(1-q) \geq \kappa \log(p)$$
then there exists a positive $C$ such that with probability larger than $1-C/p$
\begin{align*} 
d_k \simeq \sqrt{\log(p)\left[\frac{x^*_k}{nq}+\frac{\sum_{u\not = k} x^*_u}{n(1-q)}\right]} +  \frac{\log(p)\norm{x^*}_1}{n}+   \frac{\log(p)}{nq(1-q)}.
\end{align*}
\end{proposition}

\subsubsection{Assumption~\ref{as:dk} holds (proof of
  Propositions~\ref{dBern} and~\ref{dkBern})}
As shown in Appendix~\ref{scalingBern},
$$(\tA^\T (\tY-\tA x^*))_k \le T_{1,k} + T_{2,k}.$$
To derive the constant weight of Proposition~\ref{dBern}, we use a
bound on
$\|T_1\|_\infty + \|T_2\|_\infty$. To derive the non-constant
weights of Proposition~\ref{dkBern}, we use a bound on $|T_{1,k}| +
\|T_2\|_\infty$. These bounds are derived in this section.

\paragraph{Concentration of $T_2$}
Each element of the matrix $\mathbb{M}$ is a degenerate U-statistics of order $2$ of the form $2 U$ with $ U= \sum_{\ell>\ell'} g(a_{\ell,k},a_{\ell',k'})$ to which one can apply \cite{ustats}. Let us compute the different quantities involved in this concentration formula.

Since $q(1-q)\leq (q^2+(1-q)^2)/2 \leq \max(q^2,(1-q)^2),$ a deterministic upper bound of $g$ does not depend on $k,k'$ and is given by
$$A_{\mathbb{M}}=\frac{\max(q^2,(1-q)^2)}{n(n-1)q(1-q)}.$$
On the other hand for any $a\in \{0,1\}$,
 $$\E(g^2(a_{\ell,k},a))= \frac{(a-q)^2}{n^2(n-1)^2 q(1-q)}.$$
Therefore $C^2_{\mathbb{M}}= \frac{1}{2n(n-1)}$ and $$B^2_{\mathbb{M}}= \frac{\max(q^2,(1-q)^2)}{n^2(n-1)q(1-q)}.$$
Finally $D_{\mathbb{M}}$ should be chosen as an upper bound of
$$\E\left(\sum_{\ell>\ell'} g(a_{\ell,k},a_{\ell',k'}) c_\ell(a_{\ell,k})b_{\ell'}(a_{\ell',k'})\right),$$
for all choice of functions $c_\ell$, $b_{\ell'}$ such that $\E(\sum_{\ell=2}^n c_\ell(a_{\ell,k})^2) \leq 1$ and $\E(\sum_{\ell'=1}^{n-1} b_{\ell'}(a_{\ell',k'})^2) \leq 1$.
But
$$\sum_{\ell'=1}^\ell \E\left((a_{\ell',k'}-q) b_{\ell'}(a_{\ell',k'})\right)\leq \sqrt{\sum_{\ell'} \E(((a_{\ell',k'}-q)^2)}\sqrt{\sum_{\ell'} \E(b_{\ell'}(a_{\ell',k'})^2)} \leq \sqrt{n q(1-q)}.$$
By doing the same for the terms in $a_{\ell,k}$, $D_{\mathbb{M}}=\frac{n q (1-q)}{n(n-1)q(1-q)}=\frac{1}{n-1}$ works.
\revision{By \cite[Theorem 3.4.8]{GineNickl}, for all $\theta>0$, with probability larger than  $1-c_1 p^2 e^{-\theta}$, for some absolute constant $c_1$ ($c_1=5.4$ works), we have that for all $k,k'$
\begin{eqnarray*}
|\mathbb{M}_{k,k'}| &\leq & 8 C_{\mathbb{M}}\sqrt{\theta}+8\sqrt{2} D_{\mathbb{M}}\theta + 216  (B_{\mathbb{M}}\theta^{3/2}+ A_{\mathbb{M}}\theta^2) \\
&\leq &  \frac{8\sqrt{\theta}}{\sqrt{2 n (n-1)}} + \frac{8\sqrt{2}\theta}{ (n-1)} + 216 \left(\sqrt{\frac{\theta}{n} \frac{\max(q^2,(1-q)^2) \theta^2}{n(n-1) q (1-q)} }+ \frac{\max(q^2,(1-q)^2) \theta^2}{n(n-1) q (1-q)}\right)
\end{eqnarray*}
Let us take $\theta>1$ so that $\sqrt{\theta}\leq \theta$. By using the classical inequality $ab \leq (a^2+b^2)/2$, we end up with 
$$ |\mathbb{M}_{k,k'}| \leq \left(\frac{8}{\sqrt{2}}+ 8 \sqrt{2}\right) \frac{\theta}{n-1} + 108 \frac{\theta}{n}+ 108 \frac{\max(q^2,(1-q)^2) \theta^2}{n(n-1) q (1-q)}.$$}

Therefore there exists some absolute constants $c_1$ and $c_2$ such
that as soon as $n\geq 2$ \revision{and $\theta>1$}, with probability
larger than $1-c_1 p^2 e^{-\theta}$, for all $k,k'$
\begin{equation}\label{majoM}
|\mathbb{M}_{k,k'}| \leq c_2 \left(\frac{\theta}{n} + \frac{\max(q^2,(1-q)^2)}{n^2q(1-q)}\theta^2\right),
\end{equation}
\revision{where $c_2=126$ works as soon as $n\geq 20$.}
Therefore on the same event
\begin{equation}\label{majoT2}
\norm{T_2}_\infty \leq c_2 \left(\frac{\theta}{n} + \frac{\max(q^2,(1-q)^2)}{n^2q(1-q)}\theta^2\right) \norm{x^*}_1.
\end{equation}

\paragraph{Concentration around $\norm{x^*}_1$}
Since $\norm{x^*}_1$ is unknown in the previous inequality, if one wants to upper bound $T_2$, we need  to estimate it.

Applying \eqref{troisC} of Lemma \ref{concPoisGen} with $R=\ones_{\nObs\times 1}$ gives that with probability larger than $1-e^{-\theta}$,
$$\bar{x}_a=\sum_{\ell,k} a_{\ell,k} x^*_k \leq \left(\sqrt{\frac{\theta }{2}}+\sqrt{\frac{5\theta}{6}+ \sum_{\ell=1}^n Y_\ell }\right)^2.$$
But by using Bernstein's inequality, with probability larger than $1-2pe^{-\theta}$, for all $k$,
\begin{equation}\label{suma}
|\sum_{\ell=1}^n (a_{\ell,k}-q)| \leq C_{n,\theta}=\sqrt{2nq(1-q)\theta}+\max(q,(1-q))\frac{\theta}{3}
\end{equation}
Hence on this event,
$$(nq-C_{n,\theta})\norm{x}_1\leq \bar{x}_a.$$
So the first assumption on the range of $(n,q)$ is to assume that $nq>C_{n,\theta}$, which is implied by
\begin{equation}\label{nq1}
nq\geq 4 \max(q,1-q)\theta
\end{equation}
In this case, with probability larger than $1-(2p+1)e^{-\theta}$,
\begin{equation}\label{majonormx}
\norm{x}_1\leq \hat{N}_\theta:=\frac{1}{nq-C_{n,\theta}} \left(\sqrt{\frac{\theta }{2}}+\sqrt{\frac{5\theta}{6}+ \sum_{\ell=1}^n Y_\ell }\right)^2.
\end{equation}
Hence there exists some absolute constant $c_3$ such that on an event
of probability larger than $1-c_3 p^2 e^{-\theta}$,
\begin{equation}\label{T2bien}
\norm{T_2}_\infty \leq c_2 \left(\frac{\theta}{n} + \frac{\max(q^2,(1-q)^2)}{n^2q(1-q)}\theta^2\right)  \hat{N}_\theta.
\end{equation}

\paragraph{Upper-bound for $T_1$}
The upper bound on $T_2$ does not depend on $k$, but it is just a residual term. The upper bound for $T_1$ gives the main tendency and its behavior may be refined $k$ by $k$ leading to a weight $d_k$ that depends on $k$.
Recall that for fixed $k$, $T_{1,k}= R_k^\T(Y-Ax^*)$ with for all $\ell=1,\ldots,\nObs$,
$$ R_{k,\ell}= \frac{na_{\ell,k}-\sum_{\ell'=1}^\nObs a_{\ell',k}}{n(n-1)q(1-q)}.$$
By \eqref{deuxC} of Lemma \ref{concPoisGen}, on an event of probability larger than $1-2pe^{-\theta}$,
$$|T_{1,k}| \leq \sqrt{2 V_{k}^\T A x^* \theta} + \frac{\norm{R_k}_\infty \theta}{3},$$
with $V_{k,\ell}=R_{k,\ell}^2$.
But since the $a_{\ell,k}$ have values in $\{0,1\}$, one has that
$$\norm{R_k}_\infty\leq \frac{1}{(n-1)q(1-q)}.$$
Moreover,
$$V_{k}^\T A x^* \leq W \norm{x^*}_1,$$
with $$W=\max_{u,k=1,\ldots,p} w(u,k)$$ and
$$w(u,k)=\frac{1}{n^2(n-1)^2 q^2(1-q)^2}\sum_{\ell=1}^\nObs a_{\ell,u} \left(n a_{\ell,k} -\sum_{\ell'=1}^\nObs a_{\ell',k}\right)^2.$$
Combined with \eqref{majonormx}, this gives that
\begin{equation}\label{T1bien1}
\norm{T_1}_\infty \leq \sqrt{2W\theta}  \sqrt{\hat{N}} + \frac{\theta}{3(n-1)q(1-q)}
\end{equation}
This combined with \eqref{T2bien} and the choice $\theta=3\log(p)$ \revision{(which is larger than 1 since $p\geq 2$)} gives Proposition \ref{dBern}.
On the other hand, one could have applied \eqref{quatreC} of Lemma \ref{concPoisGen} to obtain that on an event of probability larger than $1-3pe^{-\theta}$,
$$|T_{1,k}| \leq \left(\sqrt{\frac{\theta }{2(n-1)^2q^2(1-q)^2}}+\sqrt{\frac{5\theta}{6(n-1)^2q^2(1-q)^2}+ V_k^\T Y }\right) \sqrt{2 \theta}+\frac{\theta}{3(n-1)q(1-q)}.$$
Once again, this combined with \eqref{T2bien} and the choice $\theta=3\log(p)$ gives Proposition \ref{dkBern}.

\subsubsection{Bounds on the $w(u,k)$'s}
First  let us remark that if we denote
$$w_1(u,k)= \frac{1}{(n-1)^2q^2(1-q)^2} \sum_{\ell=1}^\nObs a_{\ell,u}(a_{\ell,k}-q)^2$$
and
$$w_2(u,k)= \frac{1}{n^2(n-1)^2q^2(1-q)^2} \left(\sum_{\ell=1}^\nObs a_{\ell,u}\right)\left(\sum_{\ell'=1}^\nObs(a_{\ell,k}-q)\right)^2,$$
Then for all $\epsilon \in(0,1)$,
$$(1-\epsilon) w_1(u,k)+(1-\frac{1}{\epsilon})w_2(u,k) \leq w(u,k) \leq (1+\epsilon) w_1(u,k)+(1+\frac{1}{\epsilon})w_2(u,k).$$
In the sequel we consequently  need to find an upper-bound for $w_2(u,k)$ and a lower and upper bound on $w_1(u,k)$  to obtain bounds for $w(u,k)$.

\paragraph{Upper bound for $w_2(u,k)$}

By \eqref{suma} and remarking that $\max(q,1-q)\leq 1$, on an event of probability larger than $1-2pe^{-\theta}$,
\begin{align*}
w_2(u,k)\leq &  \frac{nq+\sqrt{2nq(1-q)\theta}+\frac{\theta}{3}}{n^2(n-1)^2q^2(1-q)^2} \left(\sqrt{2nq(1-q)\theta}+\frac{\theta}{3}\right)^2\\
\leq &\square \frac{n^2q^2(1-q)\theta + nq\theta^2+ (nq(1-q)\theta)^{3/2}+\theta^3}{n^4q^2(1-q)^2}\\
\leq & \square\left (\frac{\theta}{n^2(1-q)}+\frac{\theta^2}{n^3q(1-q)^2}+\frac{\theta^{3/2}}{n^{5/2}q^{1/2}(1-q)^{1/2}}+\frac{\theta^3}{n^4q^2(1-q)^2}\right)
\end{align*}
 If one assumes that
 \begin{equation}\label{nq2}
 n q(1-q) \geq \theta,
 \end{equation}
then the leading term in the previous expansion is the first one and 
\begin{equation}\label{w2bien}
w_2(u,k)\leq \square \frac{\theta}{n^2(1-q)}
\end{equation}
Now for the control of $w_1(u,k)$, if $u=k$ then on can rewrite
\begin{align*}
w_1(k,k)=& \frac{1}{(n-1)^2q^2(1-q)^2} \sum_{\ell=1}^\nObs a_{\ell,k}(a_{\ell,k}-q)^2\\
=&\frac{1}{(n-1)^2q^2(1-q)^2} \sum_{\ell=1}^\nObs (a_{\ell,k}^3-2qa_{\ell,k}^2+q^2a_{\ell,k})\\
=&\frac{1}{(n-1)^2q^2(1-q)^2} \sum_{\ell=1}^\nObs a_{\ell,k}(1-q)^2\\
=&\frac{1}{(n-1)^2q^2}  \sum_{\ell=1}^\nObs a_{\ell,k}
\end{align*}
So by \eqref{suma}, on the same event as before, because of \eqref{nq2}
\begin{equation}\label{w1kkbien}
\left|w_1(k,k)-\frac{n}{(n-1)^2q}\right|\leq \frac{\sqrt{2nq(1-q)\theta}+\frac{\theta}{3}}{(n-1)^2q^2} \leq\square \frac{(1-q)^{1/2}\theta^{1/2}}{n^{3/2}q^{3/2}}.
\end{equation}
On the other hand, if $u\not= k$, let us apply Bernstein inequality to $Z^{u,k}_\ell=a_{\ell,u}(a_{\ell,k}-q)^2$.
The expectation of $Z^{u,k}_\ell$ is given by 
$$\E(Z^{u,k}_\ell)=q^2(1-q),$$
whereas its variance is
\begin{align*}
\Var(Z^{u,k}_\ell)=& \E([Z^{u,k}_\ell]^2)-q^4(1-q)^2\\
=& \E(a_{\ell,u}^2)\E([a_{\ell,k}-q]^4)-q^4(1-q)^2\\
=&q\E(a_{\ell,k}^4-4qa_{\ell,k}^3+6q^2a_{\ell,k}^2-4q^3a_{\ell,k}+q^4)-q^4(1-q)^2\\
=&q(q-4q^2+6q^3-3q^4)-q^4(1-q)^2\\
=&q^2(1-q)(1-3q+3q^2)-q^4(1-q)^2\\
=&q^2(1-q)(1-3q+2q^2+q^3)\\
\leq &  q^2(1-q).
\end{align*}
Moreover $|Z^{u,k}_\ell|$ is bounded by 1. So Bernstein inequality gives that with probability larger than $1-2p(p-1)e^{-\theta}$,
\begin{equation}\label{monZ}
\left|\sum_{\ell=1}^n Z^{u,k}_\ell-n q^2(1-q)\right| \leq \sqrt{2 n q^2(1-q) \theta}+\frac{\theta}{3}.
\end{equation}
Hence on the same event, because of \eqref{nq2}, if we additionally assume that
\begin{equation}\label{nq3}
nq^2(1-q)\geq \theta
\end{equation}
\begin{equation}\label{w1bien}
\left|w_1(u,k)-\frac{n}{(n-1)^2(1-q)}\right| \leq \frac{\sqrt{2nq^2(1-q)\theta}+\frac{\theta}{3}}{(n-1)^2q^2(1-q)^2}\leq \square \frac{\theta^{1/2}}{n^{3/2}q (1-q)^{3/2}}.
\end{equation}
So finally there is a constant $\kappa(\epsilon)$ such that if
\begin{equation}\label{nq3bis}
nq^2(1-q)\geq \kappa(\epsilon) \theta
\end{equation}
then on this event of probability larger than $1-\square p^2 e^{-\theta}$, 
$$(1-\epsilon) \frac{1}{nq} \leq w_1(k,k)\leq (1+\epsilon) \frac{1}{nq},$$
and if $u\not =k$
$$(1-\epsilon) \frac{1}{n(1-q)} \leq w_1(u,k)\leq (1+\epsilon) \frac{1}{n(1-q)}.$$
Hence since \eqref{w2bien} holds, on the same event,
$$(1-\epsilon)^2 \frac{1}{nq} + (1-\frac{1}{\epsilon}) \square \frac{\theta}{n^2(1-q)}  \leq w(k,k)\leq (1+\epsilon)^2 \frac{1}{nq} + (1+\frac{1}{\epsilon}) \square \frac{\theta}{n^2(1-q)}.$$
 This implies up to the eventual replacement of $\kappa(\epsilon)$ by a bigger constant still depending on $\epsilon$ that
\begin{equation}\label{wkkbien} 
 (1-\epsilon)^3 \frac{1}{nq}  \leq w(k,k)\leq (1+\epsilon)^3 \frac{1}{nq},
 \end{equation}
and in the same way that for $u\not = k$ that
\begin{equation}\label{wukbien}
(1-\epsilon)^3 \frac{1}{n(1-q)}  \leq w(u,k)\leq (1+\epsilon)^3 \frac{1}{n(1-q)}.
\end{equation}

\subsubsection{Control of the constant weight (proof of Proposition~\ref{bounddBern})}
Applying \eqref{unC} of Lemma \ref{concPoisGen} with $R=\ones_{n\times1}$ gives that with probability larger than $1-e^{-\theta}$,
$$\sum_{\ell=1}^n Y_\ell\leq \square\left(\sum_{\ell,k}a_{\ell,k}x^*_k + \theta\right).$$
Then by using \eqref{suma}, we get that on an event of probability larger than $1-\square p e^{-\theta}$
$$\sum_{\ell=1}^n Y_\ell\leq \square\left( (nq+C_{n,\theta}) \norm{x^*}_1 + \theta\right).$$
This implies that on the same event
$$\hat{N} \leq \square \frac{nq+C_{n,\theta}}{nq-C_{n,\theta}} \norm{x^*}_1 + \square \frac{\theta}{nq-C_{n,\theta}}.$$
By eventually increasing $\kappa(\epsilon)$ again, we have that under \eqref{nq3bis}
$$\hat{N} \leq \square \frac{1+\epsilon}{1-\epsilon} \norm{x^*}_1 + \square \frac{\theta}{(1-\epsilon)nq}.$$
Hence, combining with \eqref{majonormx}, on an event of probability larger that $1-\square p e^{-\theta}$,
$$\norm{x^*}_1\leq \hat{N} \leq \square \frac{1+\epsilon}{1-\epsilon} \norm{x^*}_1 + \square \frac{\theta}{(1-\epsilon)nq}.$$
Hence, using again \eqref{nq3bis}, with eventually a larger $\kappa$ and fixing $\epsilon=1/2$ say, gives
$$ \square \left[\sqrt{W\theta \norm{x^*}_1} +   \frac{\theta\norm{x^*}_1}{n}+ \frac{\theta}{nq(1-q)}\right] \leq d \leq \square \left[\sqrt{W\theta \left(\norm{x^*}_1+\frac{\theta}{nq}\right)} +  \frac{\theta\norm{x^*}_1}{n}+   \frac{\theta}{nq(1-q)}\right]
$$
But by the previous computations, $W$ is of the order of $\frac{1}{n \min(q,1-q)}$, which gives Proposition \ref{bounddBern} with $\theta=3\log(p)$.

\subsubsection{Control of the non-constant weights (proof of Proposition~\ref{bounddkBern})}
Similarly, applying \eqref{unC} and \eqref{troisC} of Lemma
\ref{concPoisGen} to $V_k^\T Y$ with $V_k$ for all $k$ gives that with probability larger than $1-\square pe^{-\theta}$, for all $k$
$$V_k^\T A x^* \leq 
\left(\sqrt{\frac{\theta}{2(n-1)^2q^2(1-q)^2}}+\sqrt{\frac{5\theta}{6(n-1)^2
> q^2(1-q)^2}+V_k^\T Y}\right)^2 \leq \square\left(V_k^\T A x^* + \frac{\theta}{n^2q^2(1-q)^2}\right).$$
But $V_k^\T A x^*=\sum_{u=1}^p w(u,k)x^*_u$ which is of the order of
$$\frac{1}{nq} x^*_k + \frac{1}{n(1-q)} \sum_{u\not=k} x^*_u.$$
This gives  Proposition \ref{bounddkBern} with $\theta=3\log(p)$.

\subsection{Proof of Proposition~\ref{prop:OLS}}\label{pf:OLS}

\begin{proof}{Proposition~\ref{prop:OLS}}
By denoting $\tA_{S^*}$ the matrix of size $n\times |S^*|$  whose columns are the columns of $\tA$ corresponding to non-zero elements of $x^*$, we have for any $k\in S^*$, 
$$\xls_k=((\tA_{S^*}^H\tA_{S^*})^{-1}\tA_{S^*}^H\tY )_k=(\tG_{S^*}^{-1}\tA_{S^*}^H\tY )_k,$$
where $\tG_{S^*}=\tA_{S^*}^H\tA_{S^*}.$
Therefore, by setting $\xls_k=0$ for $k\notin S^*$, we have 
$$
\| \xls-x^*\|_2^2=\|\tG_{S^*}^{-1}\tA_{S^*}^H(\tY -\tA x^*)\|_2^2.
$$
Theorem 2.4 in \cite{mendelson2008uniform} shows that there exist constants $c_1,c_2,c_3,C > 0$ such that for any $\delta_{s^*}
\in (0,1/2)$,
if
\begin{equation}\label{BerqRIP} 
s^* \le \frac{c_1 \delta_{s^*}^2 n}{\alpha_q^4 \log(c_2 p \alpha_q^4
  / \delta_{s^*}^2 n)}, \qquad \mbox{ and } \qquad n \ge
  \frac{\alpha_q^4}{c_3 \delta_{s^*}^2} \log p
\end{equation}
where $$\alpha_q \deq \begin{cases}
\sqrt{\frac{3}{2 q(1-q)}}, & q\neq
    1/2 \\
1,& q = 1/2 \end{cases},$$
then for any $s^*$-sparse $x$,
$$(1-\delta_{s^*}) \|x\|_2^2 \le \|\tA x\|_2^2 \le (1+\delta_{s^*})\|x\|_2^2$$
with probability exceeding $1-C/p$ for a universal positive constant $C$.
Now, assume that \eqref{sparse-ber-L} is satisfied. Then,  \eqref{BerqRIP} is satisfied for $\delta_{s^*}=1/4.$ Therefore, all eigenvalues of $\tG_{S^*}$ are included in the interval  $[3/4 \ ;\ 5/4]$ and we obtain the result.
\end{proof} 


\section{Validation of assumptions for random convolution model of Section~\ref{sec:conv}} 
  
\subsection{Rescaling and recentering}
\label{app:centering}

Note that Proposition \ref{MatG} given in the next section proves in particular that $\E(\tG)=I_p$. By Lemma \ref{centering} below, we obtain in particular that $\E(\tA^{\T}(\tY-\tA x^*))=0$ as expected.
\begin{lemma}\label{centering}
Conditionally on the $U_i$'s, $\tY$ is an unbiased estimate of $\tA x^*$:
$$\expect[\tY | U_1,\ldots,U_{\nU}] = \tA x^*.$$
\end{lemma}  
\label{pf:centering}
\begin{proof}{Lemma~\ref{centering}}
We have first
\begin{align*}
\expect[\ybar] =& \frac{1}{\nU}\sum_{\ell=1}^p \expect[(Ax^*)_\ell]
= \frac{1}{\nU}\sum_{\ell=1}^p \sum_{k=1}^p \expect[a_{\ell,k} x^*_k]\\
=&\frac{1}{\nU}\sum_{\ell=1}^p \sum_{k=1}^p  x^*_k \expect[\Nb(\ell-k) ]
=\frac{1}{\nU}\sum_{k=1}^p x^*_k m
= \|x^*\|_1.
\end{align*}
The result can be now deduced:
\begin{align*}
\expect[\tY | U_1,\ldots,U_{\nU}]=&\frac{1}{\sqrt{\nU}} \left[ Ax^* -
                                     \frac{\nU-\sqrt{\nU}}{p}
                                     \sum_{\ell=1}^p x^*_\ell  \ones
                                     \right] \\
=& \frac{1}{\sqrt{\nU}} \left[ A-\frac{\nU-\sqrt{\nU}}{p}  \ones\ones^\T \right]x^*\\
=& \tA x^*.
\end{align*}
\end{proof} 

\subsection{Assumption~\ref{as:RE} holds (proof of Proposition
  \ref{GxiConv})}\label{app:MatG}
\begin{proof}{Proposition~\ref{GxiConv}}
For this purpose, let us introduce the following degenerate
U-statistics of order two, defined for all $k\in \{0,\ldots,p-1\}$ by
\begin{equation}
\label{Ustat}
\Nu(k)=\sum_{u=1}^{p}\sum_{i=1}^\nU\sum_{j\not=i, j=1}^\nU \left(\1_{U_i=u}-\frac{1}{p}\right)\left(\1_{U_j=u+k[p]}-\frac{1}{p}\right).
\end{equation}
\begin{proposition}\label{MatG}
Let $p, m > 1$ be fixed integers.
For any $k,\ell \in \{0,\ldots,p-1\}$, we have:
$$(\tG-\I_p)_{k,\ell}=\frac{1}{m}\Nu(k-\ell).$$
Furthermore, there exists absolute positive constants $\kappa$ such
that for all real number $\theta>1$ such that there exists an event
$\Omega_\Nu(\theta)$ of probability larger than $1-5.54~pe^{-\theta}$
and, on this event, for all $k \in \{0,\ldots,p-1\}$,
\begin{equation}\label{def-xi0}
|\Nu(k)|\leq  m \xi(\theta)
\end{equation}
with $$\xi(\theta)=\kappa \left(\frac{\theta}{\sqrt{p}}+\frac{\theta^2}{m}\right).$$
\end{proposition}
Note that this proves that the first part of Proposition~\ref{GxiConv}
is satisfied on the event $\Omega_\Nu(\theta)$ with $\theta = 2\log p$.
On this event,
$$\|\tG - I_p\|_{\infty} \le \xi$$
where
$$\tG := \tA^\T \tA$$
and
$$\xi = \kappa \left( \frac{\log p}{\sqrt{p}} + \frac{\log^2 p}{m} \right).$$
Thus for any $x \in \reals^p$ we have
\begin{align*}
\|x\|_2^2 =&x^\T(\tG- \tG + I_p) x\\
=& \|\tA x\|_2^2 + x^\T(I_p - \tG) x\\
\le& \|\tA x\|_2^2 + \xi \|x\|_1^2\\
\|x\|_2 \le& \| \tA x\|_2 + \sqrt{\xi} \|x\|_1
\end{align*}
and Assumption~$\ref{as:RE}({\kb,\ka})$ holds with $\ka = 1$ and $\kb
= \sqrt{\xi}$.
\end{proof}

\begin{proof}{Proposition~\ref{MatG}}
Let $\beta_0=\frac{1}{\sqrt{m}}$ and $\beta_1= \frac{\sqrt{m}-1}{p}$.
For all $k\not=\ell \in \{0,\ldots,p-1\}$,
\begin{subequations}
\begin{align}\label{AtopA}
(A^\T A)_{k,k}=&\sum_{u=1}^{p}\Nb(u)^2, \\ 
 (A^\T A)_{k,\ell}=&\sum_{u=1}^{p}\Nb(u)\Nb(u+k-\ell). 
\end{align}
\end{subequations}
First note that
\begin{align*}
\Nu(d)=&\sum_{u}\sum_{i\not = j} \1_{U_i=u}\1_{U_j=u+d}-\frac{\nU-1}{p}\sum_{j=1}^\nU \sum_u \1_{U_j=u+d}-  \frac{\nU-1}{p}\sum_{i=1}^\nU \sum_u \1_{U_i=u}+ \frac{\nU(\nU-1)p}{p^2}\\
=& \sum_{u}\sum_{i\not = j} \1_{U_i=u}\1_{U_j=u+d}-\frac{\nU(\nU-1)}{p}.
\end{align*}
If $d\not=0$,
$$\sum_{u}\sum_{i\not = j} \1_{U_i=u}\1_{U_j=u+d} = \sum_u \sum_{i, j} \1_{U_i=u}\1_{U_j=u+d}=\sum_u \Nb(u)\Nb(u+d),$$
and
\begin{equation}\label{Unot0}
\Nu(d)= \sum_u \Nb(u)\Nb(u+d) -\frac{\nU(\nU-1)}{p}.
\end{equation}
If $d=0$, if $U_i=u$ then $\sum_{j\not=i} \1_{U_j=u}=\Nb(u)-1$ and
$$\sum_{i\not = j} \1_{U_i=u}\1_{U_j=u+d}=\Nb(u)(\Nb(u)-1),$$
which leads to
\begin{equation}\label{U0}
\Nu(0)= \sum_u \Nb(u)(\Nb(u)-1) -\frac{\nU(\nU-1)}{p}= \sum_u \Nb(u)^2-\nU
-\frac{\nU(\nU-1)}{p}.
\end{equation}
Thus
$$(A^\T A)_{k,\ell} = \begin{cases}
\Nu(0)+\nU+\frac{\nU(\nU-1)}{p}, & \mbox{ if } k=\ell,\\
\Nu(k-\ell)+\frac{\nU(\nU-1)}{p}, & \mbox{ if } k\not = \ell.
\end{cases}
$$
Next 
note that
\begin{align}
\tG =&\tA^\T \tA= (\beta_0 A - \beta_1 \ones \ones^\T)^\T (\beta_0
A - \beta_1 \ones \ones^\T) \\
=& \beta_0^2 A^\T A -  \beta_0 \beta_1 (\ones \ones^\T A + A^\T
   \ones \ones^\T)+ \beta_1^2
   p \ones \ones^\T \\
\tG_{k,\ell} =& \beta_0^2 (A^\T A)_{k,\ell} - 2 \beta_0 \beta_1 m +
                \beta_1^2 p
\end{align}
For $k \neq \ell$, we have
\begin{align*}
\tG_{k,\ell} =& \beta_0^2(\Nu(k-\ell) + \frac{\nU(\nU-1)}{p}) -
                2\beta_0\beta_1 \nU + \beta_1^2p\\
=& \frac{1}{m}(\Nu(k-\ell) + \frac{\nU(\nU-1)}{p}) -
   2\sqrt{\nU}\frac{\sqrt{\nU}-1}{p} +\frac{(\sqrt{\nU}-1)^2}{p}\\
=& \frac{1}{\nU}\Nu(k-\ell).
\end{align*}
Similarly, for $k = \ell$, we have
\begin{align*}
\tG_{k,k} =& \beta_0^2(\Nu(0) + \nU+ \frac{\nU(\nU-1)}{p}) -
                2\beta_0\beta_1 \nU + \beta_1^2p\\
=& \frac{1}{\nU}\Nu(k-\ell) + 1.
\end{align*}

\medskip

For the second result, one can rewrite $\Nu(d)$ as $\Nu(d)=\sum_{i<j} g(U_i,U_j)$, with
$$g(U_i,U_j)=\sum_{u=1}^{p}\left\{\left(\1_{U_i=u}-\frac1p\right)\left(\1_{U_j=u+d}-\frac1p\right)
+\left(\1_{U_i=u+d}-\frac1p\right)\left(\1_{U_j=u}-\frac1p\right)\right\}.$$
Therefore $\Nu(d)$ is a completely degenerate $U$-statistic of order 2, and one can apply  concentration inequalities of \cite{ustats}. One can identify the corresponding constants $A_\Nu,B_\Nu,C_\Nu,D_\Nu$ as follows.
The constant $A_\Nu$ should be an upper bound of $\norm{g}_\infty$ but for $a,b\in \{0,\ldots,p-1\}$, the largest value for $|g(a,b)|$ is obtained when $b=a+d$ with $d$ such that $a=b+d[p]$ is also true. In this case, we have
$$|g(a,b)|\leq 2 \left(2\left(1-\frac1p\right)^2+ \frac{p-2}{p^2}\right) \leq 6,$$
and one can take $A_\Nu=6$.
Moreover, for all $a\in \{0,\ldots,p-1\}$,
\begin{align*}
\E(g^2(U_i,a)) \leq& 2
                     \E\left[\left(\sum_u\left(\1_{U_i=u}-\frac1p\right)\left(\1_{a=u+d}-\frac1p\right)
                     \right)^2\right] \\
&+ 2 \E\left[\left(\sum_u\left(\1_{a=u}-\frac1p\right)\left(\1_{U_i=u+d}-\frac1p\right) \right)^2\right].
\end{align*}
But
\begin{align*}
&\E\left[\left(\sum_u\left(\1_{U_i=u}-\frac1p\right)\left(\1_{a=u+d}-\frac1p\right)
  \right)^2\right]  \\
&\qquad = \E\left[\left(\left(\1_{U_i=a-d[p]}-\frac1p\right)\left(1-\frac1p\right) - \frac1p \sum_{u\not = a-d [p]}  \left(\1_{U_i=u}-\frac1p\right) \right)^2\right].
\end{align*}
Moreover the probability that $U_i=a-d[p]$ is $1/p$. Therefore, by
straightforward computations, 
 \begin{align*}
&\E\left[\left(\sum_u\left(\1_{U_i=u}-\frac1p\right)\left(\1_{a=u+d}-\frac1p\right)
  \right)^2\right]\\
&\qquad=\frac1p\left(\left(1-\frac1p\right)^2 + \frac{p-1}{p^2}\right)^2+ \left(1-\frac1p\right)\left(-\frac2p\left(1-\frac1p\right)+\frac{p-2}{p^2}\right)^2\\
&\qquad= \frac1p \left(1-\frac1p\right)^2+\frac{1}{p^2} \left(1-\frac1p\right)= \frac1p \left(1-\frac1p\right) \leq \frac1p.
\end{align*}
Therefore, 
$$\E(g^2(U_i,a)) \leq \frac{4}{p}.$$
Hence, one can choose
$$C^2_\Nu = \frac{2 \nU (\nU-1)}{p} \quad\mbox{and} \quad B^2_\Nu= \frac{4\nU}{p}.$$
Finally $D_\Nu$ is an upper bound over all functions $a_i, b_j$ such that $$\sum_{i=1}^{\nU-1} \E(a_i(U_i)^2) \leq1 \quad \mbox{and} \quad \sum_{j=2}^\nU \E(b_j(U_j)^2)\leq1$$ of
\begin{align*}
\E\left[\sum_{i<j} a_i(U_i)g(U_i,U_j)b_j(U_j)\right]=&\E\left[\sum_{i=1}^{\nU-1} a_i(U_i)\sum_{j=i+1}^\nU \E(g(U_i,U_j)b_j(U_j)|U_j)\right]\\
\leq &\E\left[\sum_{i=1}^{\nU-1} |a_i(U_i)|\sum_{j=i+1}^\nU \sqrt{\E(b_j(U_j)^2)}\sqrt{\E(g(U_i,U_j)^2|U_j)}\right]\\
\leq & \frac{2}{\sqrt{p}} \E\left[\sum_{i=1}^{\nU-1} |a_i(U_i)|\sum_{j=i+1}^\nU \sqrt{\E(b_j(U_j)^2)}\right]\\
\leq &  \frac{2\sqrt{\nU}}{\sqrt{p}}   \E\left[\sum_{i=1}^{\nU-1} |a_i(U_i)|\right]\\
 \leq &  \frac{2\nU}{\sqrt{p}}
\end{align*}
and $D_\Nu= \frac{2\nU}{\sqrt{p}}$ works. Therefore, by Theorem 3.4 of \cite{ustats}, for all $\theta>0$,
$$\P(\Nu(d)\geq c (C_\Nu \sqrt{\theta}+D_\Nu \theta + B_\Nu \theta^{3/2}+ A_\Nu \theta^2)) \leq 2.77 e^{-\theta},$$
for $c$ an absolute positive constant given in \cite{ustats,MR1857312}.  A union
bound gives the second result.

\end{proof}

\subsection{Proofs for data-dependent weights}\label{proofsddw}
Note that
\begin{align*}
\tA^\T(\tY-\tA x^*)=& \tA^\T\left[\frac{I_p}{\sqrt{m}}-\frac{\sqrt{m}-1}{pm} \ones_p\ones_p^\T\right] (Y-A x^*)\\
=&\left[\frac{A^\T}{m}-\frac{m-1}{pm} \ones_p\ones_p^\T\right] (Y-A x^*)
\end{align*}
Therefore when applying the methodology of Section \ref{sec:weights}, we identify the $\ell$th component of  $R_k$ as
$$(R_k)_\ell= \frac{\Nb(\ell-k)}{m}-\frac{m-1}{pm}.$$
Thanks to this identification one can prove the following results.

\begin{proposition}\label{ConstConv}
The constant weight given by \eqref{def:cst-W} satisfy Assumption~\asref{as:dk}{\hd} with probability larger than $1-C/p$ for some absolute positive constant $C$.
\end{proposition}
\begin{proposition}\label{upboundd}
Under the notations of Proposition \ref{ConstConv}, there exists positive absolute constants $c$ and $C$ and an event of probability larger than $1-C/p$  such that on this event
$$\hd^2 \leq  c \left( \frac{\log(p)^2}{p} + \frac{\log(p)^3}{m}\right) \left(\norm{x^*}_1 + \frac{\log(p)}{m}\right).$$
\end{proposition}

\begin{proposition}\label{NonConstConv}
The non-constant weights given by \eqref{def:noncst-W}
satisfy Assumption~\asref{as:dk}{\hd} with probability larger than $1-C/p$ for some absolute positive constant $C$.
\end{proposition}

\begin{proposition}\label{updownbounddk}
Under the notations of Proposition \ref{ConstConv}, there exists some absolute constants $\kappa_1,\kappa_2,c_1,c_2 $ and $C$ positive such that if $p\geq 5$ and if 
\begin{equation}\label{relationships}
\kappa_1 \log(p) \sqrt{p} \leq m \leq \kappa_2 p \log(p)^{-1},
\end{equation}
there exists  an event of probability larger than $1-C/p$  such that on this event
$$c_1\left( \frac{ x^*_k  \log p}{\nU} +  \frac{\log p}{p} \sum_{u\not = k} x^*_u
  + \frac{\log^2 p}{\nU^2} \right)\leq
\hd_k^2 \leq c_2\left( \frac{ x^*_k \log p }{\nU} + \frac{\log^2 p}{p} \sum_{u\not = k} x^*_u+ \frac{\log^4 p}{\nU^2}\right).$$
\end{proposition}

\subsubsection{Assumption~\ref{as:dk} holds (proof of Propositions~\ref{ConstConv} and~\ref{NonConstConv})}

Proposition \ref{NonConstConv} is just the application of \eqref{quatreC} of Lemma \ref{concPoisGen} to each of the vectors $R_k$ with $\theta=2\log(p)$.

For Proposition \ref{ConstConv} note that 
$$v_k=(R_k)_2^\T A x^* = \sum_{u=0}^{p-1} w(k-u) x^*_u.$$
Hence all the $v_k$'s satisfy that
\begin{equation}\label{vkmaj}
v_k\leq W \norm{x^*}_1
\end{equation}

But one could apply Lemma  \ref{concPoisGen} with $R=-\ones$ to obtain that
$$\P\left ( -\ybar \geq - \norm{x^*}_{1} + \sqrt{\frac{2}{\nU} \norm{x^*}_{1} \theta} + \frac{\theta}{3\nU}\right ) \leq e^{-\theta},$$
 which is equivalent to
\begin{equation}\label{normest}
 \P\left ( \norm{x^*}_{1} \geq \left [\sqrt{\frac{\theta}{2\nU}}+\sqrt{\frac{5\theta}{6\nU}+\ybar}\right ]^2\right ) \leq e^{-\theta}.
 \end{equation}

 Therefore combining \eqref{deuxC} of Lemma \ref{concPoisGen} with
 $R_k$ with \eqref{vkmaj} and \eqref{normest} leads to the desired
 result, taking $\theta=2\log(p)$.

 \subsubsection{Bounds on the $\ave{V_k,\mathbb{N}}$s}

Let $w(\ell) := \ave{V_\ell,\mathbb{N}}$.
To derive bounds on the $w(\ell)$s, we need to introduce in addition to $\Omega_\Nu(\theta)$ another event, namely $\Omega_\Nb(\theta)$.
\begin{lemma}\label{Nb}
  There exists an event $\Omega_{\Nb}(\theta)$ of probability larger
  than $1-2pe^{-\theta}$ such that on $\Omega_{\Nb}(\theta)$, for all
  $u$ in $\{0,\ldots,p-1\}$,
$$\left|\Nb(u)- \frac{\nU}{p}\right| \leq \sqrt{2\frac{\nU}{p} \theta}+ \frac{\theta}{3}.$$
\end{lemma}
This is just a classical consequence  of Bernstein's inequality to the $\nU$ i.i.d.\,variables $\1_{U_i=u}$.
Thanks to this definition, one can prove the following bounds.

\begin{lemma}\label{bornesupW}
There exists an absolute constant $c$ such that for all $\theta>1$, on the event $\Omega_\Nb(\theta)$, of probability larger than $1-pe^{-\theta}$, 
$$W\leq c \left(\frac{\theta}{p}+\frac{\theta^2}{m}\right).$$
\end{lemma}

\begin{proof}{Lemma~\ref{bornesupW}}
Recall that $W=\max w(\ell)$ with for fixed $\ell$
$$w(\ell)=\sum_{u=0}^{p-1} \frac{1}{m^2}
\left(\Nb(u)-\frac{m-1}{p}\right)^2 \Nb(u+\ell) = \ave{V_\ell,\Nb}.$$
Hence on $\Omega_\Nb(\theta)$ \
$$w(\ell)\leq  \frac{1}{m^2} \left(\sqrt{2\frac{\nU}{p} \theta}+ \frac{\theta}{3}+\frac{1}{p}\right)^2 ~~ \sum_{u=0}^{p-1}\Nb(u+\ell) \leq \square\frac{\1}{m} \left(\frac{\nU \theta}{p} + \theta^2+\frac{1}{p^2}\right)^2.$$
But $1/(p^2m)\leq \min(\theta/p, \theta^2/m)$, which gives the result.
\end{proof}

This bound can be refined for a particular range of values for $m$.
\begin{lemma}\label{bornesw}
If $p\geq 2$ and $m$ satisfies 
\begin{equation}\label{condmtheta}
5\max(2\kappa,1)\theta\sqrt{p}\leq m \leq p \theta^{-1},
\end{equation}
then there exists positive constants $c_1,c_2,c^{\prime}_1$ and $c^{\prime}_2$ such that if $\theta>3$, on $\Omega_\Nb(\theta)\cap\Omega_\Nu(\theta)$,
$$ c_1/m \leq w(0) \leq c_2/m$$
and for $\ell\not = 0$,
$$c^{\prime}_1/p \leq w(\ell) \leq c^{\prime}_2 \theta /p.$$
\end{lemma}

\begin{proof}{Lemma~\ref{bornesw}}
Let $M(\theta)= m/p +\sqrt{2\frac{\nU}{p} \theta}+ \frac{\theta}{3}$ be the bound given by Lemma \ref{Nb}.
For the upper bounds, first remark that
$$w(0)=\frac{1}{\nU^2}\sum_u \Nb(u)^3 -2 \frac{\nU-1}{p\nU^2} \sum_u\Nb(u)^2 + \left (\frac{\nU-1}{p\nU}\right )^2 \sum_u \Nb(u).$$
But $\sum_u \Nb(u)=\nU$ and on $\Omega_{\Nb}(\theta),$
\begin{align*}
\sum_u \Nb(u)^3 \leq & \sum_{u / \Nb(u) \leq 1} \Nb(u) + \sum_{u / \Nb(u) > 1} \Nb(u)^3\\
\leq & \sum_{u / \Nb(u) \leq 1} \Nb(u) + M(\theta)\sum_{u / \Nb(u) > 1} \Nb(u)^2 \\
\leq & \sum_{u / \Nb(u) \leq 1} \Nb(u)  + M(\theta)\sum_{u / \Nb(u) > 1} \Nb(u)(\Nb(u)-1) + M(\theta)\sum_{u / \Nb(u) > 1} \Nb(u)\\
\leq & \sum_{u }\Nb(u) + (M(\theta)-1) \sum_{u / \Nb(u) > 1} \Nb(u) + M(\theta)\sum_{u / \Nb(u) > 1} \Nb(u)(\Nb(u)-1) \\
\leq & \nU + (2M(\theta)-1) \sum_{u / \Nb(u) > 1} \Nb(u)(\Nb(u)-1).
\end{align*}
One can also write $\sum_u\Nb(u)^2= \nU +\sum_u\Nb(u)(\Nb(u)-1).$ Therefore
$$ w(0) \leq \frac{1}{\nU}\left(1-\frac{\nU-1}{p}\right)^2  + \frac{1}{\nU^2}\left(2M(\theta)-1-2\frac{\nU-1}{p}\right)\sum_u\Nb(u)(\Nb(u)-1).$$
But
$$\sum_u\Nb(u)(\Nb(u)-1)= \Nu(0)+ \frac{\nU(\nU-1)}{p}$$
(see (\ref{U0})).
Therefore by  Proposition~\ref{MatG} on $\Omega_{\Nb}(\theta)\cap\Omega_{\Nu}(\theta)$
\begin{equation}\label{w0}
w(0) \leq \frac{1}{\nU} \left[\left(1-\frac{\nU-1}{p}\right)^2  + \left(2M(\theta)-1-2\frac{\nU-1}{p}\right) \left(\xi(\theta)+\frac{\nU-1}{p}\right) \right].
\end{equation}
But under \eqref{condmtheta}, one has that
$$\xi(\theta)\leq 2 \kappa\frac{\theta}{\sqrt{p}}$$
and
$$M(\theta) \leq K\theta,$$
for $K$ an absolute constant large enough. Moreover,   under \eqref{condmtheta}, we observe that
$$ \frac{\theta}{\sqrt{p}} \leq \frac{\nU}{p} \leq 1/\theta \leq 1.$$
This gives
$$w(0)\leq \frac{1}{m}+ \square\frac{\theta}{p}+ \square \frac{\theta^2}{m\sqrt{p}} \leq \frac{1}{m}+ \square\frac{\theta}{p},$$
which gives the result since \eqref{condmtheta} holds.

Similarly, by using \eqref{Unot0}, for $d\not =0$, on $\Omega_{\Nb}(\theta)\cap\Omega_{\Nu}(\theta)$,
\begin{align}
\nU^2w(d) =& \sum_u \Nb(u)^2 \Nb(u+d)-2\frac{\nU-1}{p} \sum_u\Nb(u)
              \Nb(u+d) + \left (\frac{\nU-1}{p}\right )^2
              \sum_u\Nb(u)\nonumber\\ 
\leq & \left (M(\theta) -2 \frac{\nU-1}{p}\right ) \left( \Nu(d)+
        \frac{\nU (\nU-1)}{p} \right ) + \nU \left (\frac{\nU-1}{p}\right
        )^2\nonumber\\ 
\leq & \nU\left(M(\theta) -2 \frac{\nU-1}{p}\right ) \left(
        \xi(\theta)+ \frac{(\nU-1)}{p} \right ) + \nU \left
        (\frac{\nU-1}{p}\right )^2\nonumber.
\end{align}
The same simplifications lead to the upper bound for $w(d)$.

For the lower bounds, remark that
by the right hand side of \eqref{condmtheta}, $(\nU-1)p^{-1}<1/2$.
Therefore 
$$(\Nb(u)-(\nU-1)p^{-1})^2\geq (1-(\nU-1)p^{-1})^2,$$
for all $\Nb(u)\geq 1$ and therefore
$$w(0) \geq\frac{ (1-(\nU-1)p^{-1})^2}{\nU^2} \sum_{u/\Nb(u)\geq 1} \Nb(u) = \frac{ (1-(\nU-1)p^{-1})^2}{\nU}\geq \frac{1}{4\nU}.$$
If \eqref{condmtheta} is true,
\begin{align}
 \nU/5 \geq& \max(2\kappa,1) \theta \sqrt{p}\\
\geq& \kappa\theta \sqrt{p} + \kappa  \theta p p^{-1/2}\\
\geq& \kappa \theta \sqrt{p} + \kappa  \theta^2 \frac{5p}{\nU} \\
\geq& \kappa (\theta \sqrt{p}+\theta^2 p \nU^{-1}) = p \xi(\theta).
\end{align}
But, by using \eqref{Unot0}, on $\Omega_{\Nu}(\theta)$, since $(\nU-1)p^{-1}<1/3$,
\begin{align*}
\nU^2 w(d)  \geq &  \sum_u \Nb(u)^2 \Nb(u+d) - 2 \frac{\nU-1}{p}
                    \sum_u \Nb(u) \Nb(u+d) + \nU \left
                    (\frac{\nU-1}{p}\right )^2\\ 
\geq & \left( 1 - 2 \frac{\nU-1}{p}\right ) \Nu(d) + \frac{\nU(\nU-1)}{p}
       \left( 1 -  \frac{\nU-1}{p}\right )\\
\geq & -\Bigg| 1 - 2 \frac{\nU-1}{p} \Bigg| \nU \xi(\theta) + \frac{\nU(\nU-1)}{p} \left( 1 -  \frac{\nU-1}{p}\right ) \\
\geq & \square \nU \left( \frac{\nU}{4p}-\xi(\theta)\right)\\
\geq & \square\frac{\nU^2}{20p}.
\end{align*}
\end{proof}

\subsubsection{Control of the constant weight (proof of Proposition \ref{upboundd})}
Let $\theta>1$. First remark that \eqref{unC} with $R=\ones$ gives that with probability larger than $1-e^{-\theta}$
$$\bar{Y} \leq \square \left[\norm{x^*}_1+\frac{\theta}{m}\right].$$
Moreover using Lemma \ref{Nb}, on $\Omega_\Nb(\theta)$,
$$B\leq \square \left[\sqrt{\frac{\theta}{mp}}+\frac{\theta}{m}\right].$$
Combining this with Lemma \ref{bornesupW} and taking $\theta=2\log(p)$ gives
\begin{align*}
\hd^2 \leq & \square \left[ W \theta \left(\norm{x^*}_1+\frac{\theta}{m}\right) + \theta^2 \left(\frac{\theta}{mp}+\frac{\theta^2}{m^2}\right)\right]\\
\leq & \square \left[ \left(\frac{\theta}{p}+\frac{\theta^2}{m}\right) \theta \left(\norm{x^*}_1+\frac{\theta}{m}\right) + \theta^2 \left(\frac{\theta}{mp}+\frac{\theta^2}{m^2}\right)\right],
\end{align*}
which implies the result.
\subsubsection{Control of the non-constant weights (proof of Proposition \ref{updownbounddk})}

Let $\theta=2\log(p)$ (since $p\geq 5,$ this  ensures that $\theta>3$). Applying \eqref{quatreC} of Lemma \ref{concPoisGen} to $(R_k)_2$ gives that with probability larger than $1-pe^{-\theta}$,
$$\hat{v}_k = \ave{V_k,Y} \leq \square \left[v_k + B^2\theta\right].$$
But since 
$$v_k=\sum_u w(k-u) x^*_u,$$
one can use Lemma \ref{bornesw} (by choosing $\kappa_1, \kappa_2$ such that \eqref{condmtheta} holds) to show that
\begin{align*}
\hd_k^2  \leq & \square \left[v_k\theta+B^2\theta^2\right] \\
 \leq &  \square \left[ \frac{x^*_k\theta}{m} + \sum_{u\not = k} x^*_u \frac{\theta^2}{p} + \frac{\theta^4}{m^2}\right],
\end{align*}
since $\theta^2/m \geq \theta / p$.

For the lower bound, the arguments are similar
\begin{align*}
\hd_k^2  \geq & \square \left[v_k\theta+B^2\theta^2\right] \\
 \geq &  \square \left[ \frac{x^*_k\theta}{m} + \sum_{u\not = k} x^*_u \frac{\theta}{p} + B^2 \theta^2\right],
\end{align*}
but since $(m-1)/p <1/3$ and since there is at least one $\Nb(u)\geq 1$ for some $u$, then $B>2/3 m^{-1}$. Hence
$$\hd_k^2 \geq \square \left[ \frac{x^*_k\theta}{m} + \sum_{u\not = k} x^*_u \frac{\theta}{p} + \frac{\theta^2}{m^2}\right],$$
which gives the result.

\subsection{Proof of Proposition~\ref{prop:OLS2}}
\begin{proof}{Proposition~\ref{prop:OLS2}}
The first part of the proof follows the proof of
Proposition~\ref{prop:OLS}, yielding
$$
\| \xls-x^*\|_2^2=\|\tG_{S^*}^{-1}\tA_{S^*}^H(\tY -\tA x^*)\|_2^2.
$$
By Proposition~\ref{GxiConv}, the maximum eigenvalue of
$\tG^{-1}_{S^*}$ is bounded by $\frac{1}{1-\sqrt{s^* \xi}} \le c^{\prime}$
  under \eqref{sparse-conv-L}, yielding the result.
\end{proof}

\end{document}